\documentclass{mathscan}

\newtheorem{step}{Step}

\newtheorem{question}{Question}[section]

\newtheorem{theorem}{Theorem}[section]
\newtheorem{proposition}[theorem]{Proposition}
\newtheorem{corollary}[theorem]{Corollary}
\newtheorem{lemma}[theorem]{Lemma}

\theoremstyle{definition}
\newtheorem{remark}[theorem]{Remark}
\newtheorem{definition}[theorem]{Definition}
\newtheorem{example}[theorem]{Example}

\numberwithin{equation}{section}

\begin{document}

\title[Parameter-dependent analyticity theorems]{Analyticity theorems for parameter-dependent plurisubharmonic functions}

\author{Bojie He}
\address{School of Mathematical Sciences\\
Peking University\\
 No.5 Yiheyuan Road Haidian District\\
100871, Beijing\\
China}
\email{hbj@amss.ac.cn}


\begin{abstract}
  In this paper, we first show that a union of upper-level
  sets associated to fibrewise Lelong numbers of plurisubharmonic
  functions is in general a pluripolar subset. Then we obtain analyticity theorems for a union of sub-level
  sets associated to fibrewise complex singularity exponents of some special (quasi-)plurisubharmonic
  functions. As a corollary, we confirm that, under certain conditions, the logarithmic poles of relative Bergman kernels form an analytic subset when the (quasi-)plurisubharmonic weight function has analytic singularities. In the end, we give counterexamples to show that the
  aforementioned sets are in general non-analytic even if the plurisubharmonic function is supposed to be continuous.
\end{abstract}

\maketitle

\section{Introduction and main results}
\subsection{Background: analyticity theorems associated to Lelong numbers}
Plurisubharmonic functions play crucial roles in the study of several complex
variables and complex geometry. Let $\varphi = \varphi (z, w)$ be a
plurisubharmonic function defined on a product of two unit polydiscs
$\Delta^n_z \times \Delta^m_w$ with coordinates $z = (z_1, \ldots, z_n)$ and
$w = (w_1, \ldots, w_m)$, $c > 0$ be any real number. Let us consider the
following two types of upper-level sets of $\varphi$ defined by
\[ E_c (\varphi) := \{ (z, w) \in \Delta^n_z \times \Delta^m_w ; \nu_{(z,
   w)} (\varphi) \geq c \} \]
and
\[ X_c (\varphi) : = \{ (z, w) \in \Delta^n_z \times \Delta^m_w ; \nu_z
   (\varphi_w) \geq c \} . \]
Here $\nu_x (\varphi)$ denotes the Lelong number (cf. Definition \ref{def21})
of $\varphi$ at the point $x$ and $\varphi_{w_0} := \varphi
(\cdot, w_0)$. By the basic restriction property of Lelong numbers, it is clear that $E_c (\varphi)$ is contained in $X_c (\varphi)$.

The celebrated Siu's analyticity theorem (see {\cite{siu2}}) tells us that $E_c
(\varphi)$ is an analytic subset (cf. Definition \ref{def2.1}) of $\Delta^n_z \times \Delta^m_w$. An
equivalent formulation is that $\nu_{(z,w)} (\varphi)$, viewed as a function of $(z,w)$,
is upper semi-continuous with respect to the analytic Zariski topology. Later,
Demailly (see {\cite{demailly-book-hep}}) and Kiselman (see {\cite{kiselman}}) has applied distinct approximation and attenuation skills to recovering Siu's result.
In {\cite{guan-zhou1}}, the authors have also given a rather new approach to Siu's analyticity theorem.

It is very natural to ask whether $X_c (\varphi)$ becomes an analytic subset
in $\Delta^n_z \times \Delta^m_w$. Unfortunately, counterexamples of
discontinuous plurisubharmonic functions (see Example 4.3 in {\cite{wangxiaoqin}})
actually show that this statement fails in general.

However, if we suppose that $e^{\varphi}$ is locally H{\"o}lder continuous
with respect to $w$, i.e. for any compact $K\subset \Delta^n\times \Delta^m$ and $(z, w_1), (z, w_2) \in K$,
\begin{equation}
  | e^{\varphi (z, w_1)} - e^{\varphi (z, w_2)} | \leq C_K |
  w_1 - w_2 |^{\alpha_K} \label{1.1}
\end{equation}
holds for some $\alpha_K, C_K > 0$, then an analyticity theorem could be obtained as follows.

\begin{theorem}
  \label{se}{{{{{{{(Theorem 4.10 in
  {\cite{wangxiaoqin}})}} }}}}}If $e^{\varphi}$ is locally H{\"o}lder continuous
with respect to $w$, then $X_c
  (\varphi)$ is an analytic subset of $\Delta^n_z \times \Delta^m_w$.
\end{theorem}

One main example of locally H{\"o}lder continuous plurisubharmonic functions
is a plurisubharmonic $\varphi$ with analytic singularities (see (\ref{analyticsing})) or generalized analytic singularities (cf. Example 2.4 in \cite{demailly-kollar}).
More generally, the
authors in {\cite{wangxiaoqin}} imposed a weaker condition on $\varphi$ which is
called ``upper H{\"o}lder regular with respect to $w$'' to get Theorem 1.1. Essentially, their proof of Theorem \ref{se} is based on Kiselman's
attenuation techniques and Kiselman's minimum principle (see {\cite{kiselman}},
{\cite{kis}}).

\subsection{Main results: analyticity theorems associated to complex singularity exponents}
Now we change the subject to be studied into complex singularity exponent (cf. Definition
\ref{def21}), which reflects the
singularities of plurisubharmonic functions in a slightly different way. \ The complex singularity
exponent $c_z (\varphi)$ is also considered as an analogue of the log canonical threshold in
birational geometry.

Let $\varphi$ be a
plurisubharmonic function defined on
$\Delta^n_z \times \Delta^m_w$ . Set 
\[ F_c (\varphi) := \{ (z, w) \in \Delta^n_z \times \Delta^m_w ; c_{(z,
   w)} (\varphi) \leq c \} \]
and
\[ Y_c (\varphi) := \{ (z, w) \in \Delta^n_z \times \Delta^m_w ; c_z
   (\varphi_w) \leq c \} \]
to be the related subset.

Thanks to H{\"o}rmander-Bombieri-Skoda $L^2$ extension theorem (Theorem 4.4.4 in
{\cite{hormanderbook}} or p. 384 in {\cite{demailly-book}}) and the restriction
property of complex singularity exponents, $F_c (\varphi)$ becomes an analytic subset of $\Delta^n_z \times
\Delta^m_w$ and $F_c (\varphi) \subset Y_c (\varphi)$. In general, in view of
the basic fact that complex singularity exponent is the inverse of the Lelong
number in the complex dimension one case, Example 4.3 in {\cite{wangxiaoqin}} also
implies that $Y_c (\varphi)$ may not be an analytic subset of $\Delta^n\times\Delta$ for general
$\varphi$. Nevertheless, we can still prove

\begin{theorem}
  [{{{{{{{Main theorem
  1}}}}{{}}}}}] Let $\varphi$ be a plurisubharmonic function defined on $\Delta_z^n\times\Delta^m_w$, then $Y_c (\varphi)$ is a complete pluripolar subset of
  $\Delta^n_z \times \Delta^m_w$. Furthermore, $X_c (\varphi)$ is a pluripolar
  subset of $\Delta^n_z \times \Delta^m_w$.
  \label{polar}
\end{theorem}

Let us recall that a set is called {{pluripolar}} if locally it is
contained in a subset where a locally integrable plurisubharmonic function
equals $- \infty$; a set is called {{complete pluripolar}} if locally
it is exactly the set of points where a locally integrable plurisubharmonic
function equals $- \infty$ (see Definition \ref{def2.1}). Moreover, in comparison with Theorem \ref{se}, it is also natural to ask

\begin{question}\label{xiaochou}
  \label{ss}If $e^{\varphi}$ is locally H{\"o}lder continuous with respect to
  $w$, will $Y_c (\varphi)$ become an analytic subset?
\end{question}

Before stating our main result, it is worthy to mention the following
consequence which was first obtained by Varchenko ({\cite{varchenko}}), and then proved via an analytic approach by
Demailly and Koll{\'a}r ({\cite{demailly-kollar}}):

\begin{theorem}
  \label{dk}
  Let $\varphi$ be a plurisubharmonic function defined on $\Delta_z^n\times\Delta^m_w$. Then
  \begin{enumerate}
    \item \label{jj}{{(Lemma 3.2 in {\cite{demailly-kollar}})}} Suppose
    that $e^{\varphi}$ is {{H{\"o}lder }}continuous with respect to $w$
    on $\Delta^n_z \times \Delta^m_w$, then for any $z_0 \in \Delta^n$, $\{ z
    = z_0 \} \bigcap Y_c (\varphi)$ is a closed subset of $\Delta^m$.
    
    \item \label{b}{{({\cite{varchenko}} or Theorem 3.1 in
    {\cite{demailly-kollar}})}} If $\varphi = \log | f | 
    $ for some $f \in \mathcal{O} (\Delta^n_z \times \Delta^m_w)$,
    then for any $z_0 \in \Delta^n$, $\{ z = z_0 \} \bigcap Y_c (\varphi)$ is an
    analytic subset of $\Delta^m$.
  \end{enumerate}
\end{theorem}

\begin{remark}
 Analyticity of a set $E$ along both vertical and horizontal
  directions doesn't mean the whole analyticity of $E$. A counterexample is $$E
  : = \{ ( n^{- 2} {, n^{- 2}}  ) ; n \in \mathbb{N} \} 
  \bigcup \{(0, 0)\}.$$ Therefore, in the case of $(\ref{b})$, we cannot obtain the analyticity of $Y_c (\varphi)$ directly.
\end{remark}

It will be easy to check that in (\ref{jj}) H{\"o}lder continuity along the
$w$-base direction is enough for their proof. Besides, their proof is based on
the $L^2$ extension theorem ({\cite{ohsawa-takegoshi}}) and an idea of Angehrn
and Siu ({\cite{as}}). Note (\ref{jj}) provides a new proof of the
semi-continuity Theorem \ref{dk} (\ref{b}) (cf. section 3 in
{\cite{demailly-kollar}}).

Since (\ref{jj}) is only concerned with the closeness of $Y_c (\varphi)$ along
the base direction (note analyticity along the $z$-direction is obvious!),
it will be intuitive for us to speculate the whole analyticity of $Y_c
(\varphi)$. The following result gives an affirmative answer to Question \ref{xiaochou} in the special case that $m=1$ and $\varphi$ has analytic singularities.

\begin{theorem}
 [{{{Main theorem 2}}}] Let $\varphi$ be a plurisubharmonic function defined on $\Delta_z^n\times\Delta_w$ \emph{(i.e. $m=1$)}. Assume moreover that $\varphi$ has analytic singularities, then $Y_c (\varphi)$ is an analytic subset.
   \label{mainresult}
\end{theorem}

It is clear that Theorem \ref{mainresult} implies (\ref{b}) in
Theorem \ref{dk} in the case when $m=1$ immediately. Note the proof of (\ref{b}) in Theorem \ref{dk} is essentially reduced to considering an one-dimensional base by a stratification process (cf. \cite{demailly-kollar}).

The proof of Theorem \ref{mainresult} is obtained via an analytic method, which could precisely point out
the defining functions of the analytic subset $Y_c (\varphi)$ when
$m = 1$ (cf. Remark \ref{jinaj}). The method will be a perfect combination of the $L^2$ extension theorem
({\cite{ohsawa-takegoshi}}), the restriction techniques occurred in
{\cite{guenancia2}}, the generalized Siu's lemma (cf. {\cite{zhouzhu2}},
{\cite{zhouzhu}}) and an improvement of the stability of integrals of the following forms ($\delta > 0$)
$$
  \int \frac{| f (z, w) |^2  }{(| g_1 (z, w) |^2 + \cdots
  + | g_N (z, w) |^2)^{\delta / 2}   } d V (z)
$$
on $\Delta^n_z \times \Delta _w$ as discussed in {\cite{phongsturm}}. Our
improvement (cf. Theorem \ref{useful}) is not merely considering essentially the $f \equiv 1$ case as in
{\cite{phongsturm}}; instead, we replace the original plurisubharmonicity condition on $-\log R$ by a new integrable condition (\ref{219}).

Furthermore, following an algebraic approach, we can also establish a parameter dependent analyticity theorem associated to \emph{proper} holomorphic surjective submersions, which in particular asserts that triviality of restricted multiplier ideals is an open condition in the analytic Zariski topology under certain assumptions.
\begin{theorem}[{{Main theorem 3}}]\label{third}
Let $X$ be a projective manifold, $Y$ be a compact complex manifold and $f:X\rightarrow Y$ be a surjective holomorphic submersion of relative dimension $n$. Let $\varphi$ be a quasi-psh function on $X$ with analytic singularities. Then
$$
Z_c (\varphi):=\bigcup_{y\in Y}\{p\in X_y;c_{p}(\varphi |_{X_y})\leq c\}
$$
is an analytic subset (\emph{i.e.} a finite union of algebraic subvarieties via Serre's \emph{GAGA} principle) of $X$. In particular, if moreover $\mathcal{I}(\varphi|_{X_{y_0}})=\mathcal{O}_{X_{y_0}}$ for some $y_0 \in Y$, then 
\begin{equation}\label{sulia}
\mathcal{I}(\varphi)|_{X_y}=\mathcal{I}(\varphi |_{X_y})=\mathcal{O}_{X_y}
\end{equation}
holds for general $y\in Y$ (i.e. outside some nowhere dense analytic subset in $Y$), where the multiplier ideal sheaf $\mathcal{I}(\varphi)$ is defined as in \emph{(\ref{wocaof})}.
\end{theorem}
\begin{remark}\label{jiaowai1}
It is not hard to see that the first identity in (\ref{sulia}) already holds for general $y\in Y$ under the assumption on $X,Y,f$ and $\varphi$ in the above theorem. We give some details for the reason as follows.

First note that by the main result of \cite{var} (see also Theorem 3.1 in \cite{dem}), $Y$ is K\"{a}hler. Let $L$ be a positive line bundle over $X$. By B. Berndtsson's theorem on positivity of direct images (cf. Theorem 1.2 in \cite{ber}), we moreover see that $Y$ admits a (strictly) Nakano positive vector bundle $f_{\ast}(K_{X/Y}+L)$. Hence $Y$ is projective since $\det f_{\ast}(K_{X/Y}+L)$ is a positive line bundle. Consequently, $f$ is induced by an analytification of a smooth morphism $F:X\rightarrow Y$ (by Serre's GAGA principle). Since $X$ is projective, we see that $\mathcal{I}(\varphi)$ is actually the analytification (which is sometimes called $\cdot^{\textbf{An}}$) of an algebraic coherent sheaf $\mathcal{J}(X,\mathfrak{a}^{c})$ for some $c>0$, where $\mathfrak{a}$ is an integrally closed (algebraic) ideal sheaf corresponding to $\varphi$ (e.g. first fix some $c\in\mathbb {R}_{>0}$ and take finite local generators $\{f_i\}$ of $\varphi/c$, then set $\mathfrak{a}^\textbf{An}\subset \mathcal{O}^\textbf{An}_X$ to be the unique integrally closed ideal sheaf generated by $\{f_i\}$ and $\mathfrak{a}\subset\mathcal{O}_X$ to be the unique algebraic coherent sheaf associated to  $\mathfrak{a}^\textbf{An}$ thanks to GAGA again). Moreover, $\mathfrak{a}|_{X_y}\neq 0$ for general $y\in Y$ because $F$ is proper. Finally, applying Theorem \ref{imp} and Remark \ref{rmk213} to $F:X\rightarrow Y$ and $\mathfrak{a}$ we confirm our claim.
\end{remark}

\begin{remark}\label{supe}
In the case when $f$ is not necessarily a submersion and $\varphi$ has arbitrary singularities, it follows from the proof in \cite{xiamingchen} that the first identity in (\ref{sulia}) holds outside some pluripolar subset in $Y$. 
\end{remark}

Compared with Remark \ref{jiaowai1} and \ref{supe}, the improvement in Theorem \ref{third} here, is that in the special case when $\mathcal{I}(\varphi|_{X_{y_0}})=\mathcal{O}_{X_{y_0}}$ for some $y_0 \in Y$, we precisely point out the set outside which (\ref{sulia}) holds can be taken as $f(Z_1 (\varphi))$, which is an analytic subset in $Y$ according to Remmert's proper mapping theorem. Since the last identity implies the first identity in (\ref{sulia}) by restriction formula, we know a priori, under the assumption of Theorem \ref{third}, that the subset of $Y$ in which (\ref{sulia}) holds is open in the general topology (cf. Remark \ref{spod}).
\subsection{Applications: analyticity of the set of zeroes of relative Bergman kernels}
As an immediate application of Theorem \ref{mainresult}, we can first state the following corollary.
\begin{corollary}\label{kias}
  \label{cor1.6}Suppose $\varphi$ has analytic singularities on $\Delta^n_z
  {\times \Delta_w} $, then the complete pluripolar set
$$
    \{(z, w) \in \Delta^n_z {\times \Delta_w}  {; \log K_{c \varphi}}  (z,
    w) = - \infty\}
$$
  is also an analytic subset for any $c > 0$, where the relative Bergman
  kernel ${K_{c \varphi}}  (z, w)$ is defined as in {{(\ref{28})}}.
\end{corollary}
It will be interesting to explore whether the psh function $\log K_{c\varphi}(z,w)$ has analytic singularities in Corollary \ref{kias}. Another interesting question will be to investigate the case in Corollary \ref{kias} when the total manifold is compactly fibered over the base manifold. Motivated by this, we can then state the second corollary as follows, which is obtained by Theorem \ref{third}.

\begin{corollary}\label{second}
Let $X$ be a projective manifold, $Y$ be a compact complex manifold and $f:X\rightarrow Y$ be a surjective holomorphic submersion of relative dimension $n$. Let $A\rightarrow X$ be a very ample line bundle and $\varphi$ be a quasi-psh function on $X$ with analytic singularities so that $2i\partial\overline{\partial}\varphi\geq -C \cdot i\Theta_{h_A}(A)$ for some smooth positive metric $h_A$ on $A$. Then there exists a positive number $N=N(C,n)=\lceil C\rceil +n+1$ (where $"\lceil\cdot\rceil"$ means round up), such that the complete pluripolar set $$S_{N,A}:=\{x\in X; \log K_{X/Y, NA, h_A^{N}e^{-2\varphi}}(x)=-\infty\}$$ is also an analytic subset of $X$, where $K_{X/Y, NA, h_A^{N}e^{-2\varphi}}$ denotes the relative Bergman kernel of $K_{X/Y}+NA$ with respect to the metric $h_A^{N} e^{-2\varphi}$ on $NA$ (see (\ref{chaoshi})).

\end{corollary}

\subsection{Examples and organizations}

In the end, we will recall C. Li's counterexample ({\cite{lichi}}) to show
that $Y_c (\varphi)$ is in general not an analytic subset even if we assume
$\varphi$ to be{{ continuous}} (which precisely means $e^{\varphi}$ is
continuous). Note the plurisubharmonic function appeared as an counterexample in
{\cite{wangxiaoqin}} (cf. Example \ref{buyong} below) is discontinuous.

The rest of the paper is organized as follows.
In section \ref{sec2}, we list some basic lemmas and propositions for the
proof of our main results.
In section \ref{sec3}, we show proofs of our main results and applications.
In section \ref{sec6}, we give the explicit construction of the aforementioned
counterexamples.

\section{Preparations}\label{sec2}

\subsection{Singularities of plurisubharmonic functions}
Let us first recall some basic notions to measure the singularities of plurisubharmonic functions.

\begin{definition}
  \label{def2.1}Let $\varphi$ be a plurisubharmonic function defined on a
  domain $\Omega \subset \mathbb{C}^n$.
  \begin{enumerate}
    \item A subset $A \subset \Omega$ is called a \textbf{pluripolar
    subset}, if for any $x \in \Omega$, there exists a neighborhood $U$ of
    $x$ and $\psi \in {PSH} (U)$ and $\psi\not\equiv -\infty$, such that
    $A \bigcap U \subset \{\psi = - \infty\}$;
    
    \item A subset $A \subset \Omega$ is called a \textbf{complete
    pluripolar subset}, if for any $x \in \Omega$, there
    exists a neighborhood $U$ of $x$ and $\psi \in {PSH} (U)$ and $\psi\not\equiv -\infty$, such that $A \bigcap U = \{\psi = - \infty\}$;
    
    \item A subset $A \subset \Omega$ is called an \textbf{analytic
    subset}, if $A$ is closed and for any $x \in \Omega$, there exists a
    neighborhood $U$ of $x$ and $F_j \in \mathcal{O} (U)$, $j = 1, \ldots .,
    M$, such that $A \bigcap U = \bigcap_{j = 1}^M F_j^{-1}(0)$.
  \end{enumerate}
\end{definition}

It is easy to see that $A$ is an analytic subset $\Rightarrow$ $A$ is a complete
pluripolar subset ${\Rightarrow}$ $A$ is a pluripolar subset. An equivalent definition of an
analytic subset $A$ is that for any $x \in \Omega$,
$(A, x)$, which means the germ of $A$ at $x$, is analytic (i.e. the germs of zeroes of a family of holomorphic
functions).

\begin{definition}
  \label{def21}Let $\varphi$ be a (quasi-)plurisubharmonic function defined on a
  domain $\Omega \subset \mathbb{C}^n$. For any $x \in \Omega$, the stalk of
  the (analytic) multiplier ideal sheaf $\mathcal{I}(\varphi)$ associated to $\varphi$ at $x$ is
  defined as
  \begin{equation}\label{wocaof}
    \mathcal{I} (\varphi)_x := \{f \in \mathcal{O}_{\Omega, x} ; |f|^2 e^{- 2
    \varphi} \in L^1_{\text{loc}} \text{ near } x\} .
  \end{equation}
  The complex singularity exponent of $\varphi$ at $x$ is defined as
$$
    c_x (\varphi) := \sup \{ c \geq 0 ; e^{- 2 c \varphi} \in
    L^1_{\text{loc}} \text{ near } x \}=\sup \{ c \geq 0 ;\mathcal{I}(c\varphi)_x=\mathcal{O}_{\Omega,x} \}.
$$
  The Lelong number of $\varphi$ at $x$ is defined as
  \begin{equation}
    \nu_x (\varphi) := \liminf_{z \rightarrow x} \frac{\varphi (z)}{\log
    |z - x|} .
  \end{equation}
\end{definition}

We say that a psh or a quasi-psh function $\varphi$ has{\textbf{ analytic singularities}}, if near any $x \in
\Omega$,
\begin{equation}\label{analyticsing}
  \varphi = \frac{\alpha}{2} \log ( \sum_{i = 1}^N |f_i |^2 ) + \mathcal{C}^\infty,
\end{equation}
where $\alpha \in \mathbb{R}_+$ and each $f_i$ is a holomorphic function near $x$. 
It it clear that all plurisubharmonic functions
with analytic singularities are locally H{\"o}lder continuous. When $\varphi$ has analytic singularities, the (analytic) multiplier ideal at $x$ coincides with the multiplier ideal associated to the ideal generated by $\{f_{i,x}\}_{i=1}^N$.

Let $\mathcal{H}_{\Omega}(c\varphi)$ be the Hilbert space containing  all $L^2$ holomorphic functions on $\Omega$ with respect to the measure $e^{-2c\varphi}d\lambda$, where $d\lambda$ is the Lebesgue measure.

\begin{proposition}
  \label{213}Let $\Omega \subset \subset \mathbb{C}^n$ be a pseudoconvex
  domain and $\varphi$ a plurisubharmonic function on $\Omega$. Then for any $c>0$,
$$
    \{ z \in \Omega ; c_z (\varphi) \leq c \} = \{ z \in \Omega ; e^{- 2 c
    \varphi} \notin L^1_{\emph{loc}, z} \} = \bigcap_{f \in \bigcup_{c' > c}
    \mathcal{H}_{\Omega} (c' \varphi)} f^{- 1} (0) .
$$
\end{proposition}

\begin{proof}
  Let us first show the second identity. For the ``$\subset$'' direction, we
  assume that there exists $c' > c,$ $f \in \mathcal{H}_{\Omega} (c' \varphi)$
  such that $f (z) \neq 0$ at $z \in \Omega$. Then it is clear that $e^{- 2 c'
  \varphi}$ is $L_{{loc}}^1$ near $z$ and $c < c' \leq c_z (\varphi)$.
  For the other "${\supset}$" direction, using H{\"o}rmander-Bombieri-Skoda's
  extension theorem (e.g. Theorem 4.4.4 in {\cite{hormanderbook}} or Theorem
  7.6 in {\cite{demailly-book}}) we know that if $e^{- 2 c \varphi} $ is
  $L^1_{{loc}}$ near $z \in \Omega$, then there exists $f \in
  \mathcal{H}_{\Omega} (c \varphi)$ and $f (z) = 1$, which confirms our claim.
  The first identity is nothing but a direct consequence of the openness
  property of multiplier ideal sheaves ({\cite{bern13}}).
\end{proof}

\begin{remark}
  By the same argument of the proof, we can also show that
  \begin{equation}
    \{ z \in \Omega ; c_z (\varphi) \leq c \} = \bigcap_{f \in
    \mathcal{H}_{\Omega} (c \varphi)} f^{- 1} (0). \label{2626}
  \end{equation}
  Hence the openness property of multiplier ideal sheaves (cf. {\cite{bern13}} or
  {\cite{guan-zhou15soc}}) can be equivalent to saying that two analytic subsets $$\bigcap_{f \in \mathcal{H}_{\Omega} (c \varphi)} f^{- 1} (0)$$ and
  $$\bigcap_{f \in \bigcup_{c' > c} \mathcal{H}_{\Omega} (c' \varphi)} f^{- 1}
  (0)$$ coincide. The openness property also yields that
$$
    \{ z \in \Omega ; c_z (\varphi) \leq c \} = {Supp}
    (\mathcal{O}_{\Omega} / \mathcal{I} (c \varphi)) .
$$
\end{remark}

\subsection{Ohsawa-Takegoshi $L^2$ Extension theorem and Restriction formula}

In order to prove our main results, we will adopt the following version of
Ohsawa-Takegoshi $L^2$ extension theorem:

\begin{lemma}
 [cf. {{{\cite{demailly-book-hep}}\label{ote}}}]Let $\varphi (z, w)$ be a
  plurisubharmonic function on $\Delta^n_z \times \Delta^m_w$, then for any $f
  \in \mathcal{O} (\Delta^n_z)$ satisfying that
$$
    \int_{\Delta^n_z} | f (z) |^2 e^{- \varphi (z, 0)} d \lambda_n < \infty,
$$
  there exists an $F \in \mathcal{O} (\Delta^n_z \times \Delta^m_w)$ such that
  $F (z, 0) = f (z)$ and
$$
    \int_{\Delta^n_z \times \Delta^m_w} \frac{| F |^2 e^{- \varphi} 
    }{| w |^{2 m} \log^2 | w |   
    } d \lambda_{m + n} \leq C \int_{\Delta^n_z} |f (z) |^2 e^{-
    \varphi (z, 0)} d \lambda_n ,
$$
  where $C = C_m$ is a numerical constant which only depends on $m$.
\end{lemma}

As a very useful application of $L^2$ extension theorem, we recall the following
restriction formula of multiplier ideal sheaves.

\begin{theorem}
  [{{\emph{Theorem 14.1 in {\cite{demailly-book-hep}}}}}] \label{resc}
  Let $\varphi$ be a
  plurisubharmonic function on a complex manifold $X$, and let $Y \subset X$
  be a submanifold. Then
$$
    \mathcal{I} (\varphi |_Y) \subset \mathcal{I} (\varphi) |_Y .
$$
\end{theorem}

\subsection{Log plurisubharmonicity of fibrewise Bergman kernels}

Let $\Delta^n\times \Delta^m$ be the product of unit polydiscs and $\varphi$ be a plurisubharmonic function defined on $\Delta^n\times\Delta^m$. Let $\mathcal{H}_{\Delta^n} (c \varphi_w)$ denote the Hilbert space which
consists of all $L^2$ holomorphic functions on $\Delta^n$ with respect to the
measure $e^{- 2 c \varphi_w} d \lambda_n$. For any $c > 0$, let $K_{c \varphi}
(z, w)$ denote the fibrewise Bergman kernel with respect to the weight $e^{-
2 c \varphi_w}$. By definition, it can be expressed as an extremal function
\begin{equation}
  K_{c \varphi} (z, w) = \sup \{ |f_w (z) |^2 ; f_w \in
  \mathcal{H}_{\Delta^n} (c \varphi_w),|| f_w ||\leq 1 \} , \label{28}
\end{equation}
where $||f_w||^2=\int_{\Delta^n}|f_w(z)|^2 e^{-2c\varphi(z,w)}d\lambda_n (z)$. Berndtsson has proved the following well-known result.

\begin{theorem}
  [cf. {{{\cite{berndtsson2}}}}]\label{2.4}$\log K_{c \varphi} (z, w)$ is plurisubharmonic with respect to both $z$ and $w$.
\end{theorem}

Now we turn to the case when the fiber is compact. Let $X$ be a weakly pseudoconvex K\"{a}hler manifold and $Y$ be a connected complex manifold. Let $f:X\rightarrow Y$ be a proper surjective holomorphic submersion, $L$ be a holomorphic line bundle over $X$ equipped with a singular psh metric $h_L$ so that

(i) the curvature current $i\Theta_{h_L}(L)\geq 0$;

(ii) $H^0(X_{y_{0}}, \mathcal{O}(K_{X_{y_0}}+L)\otimes \mathcal{I}(h_L |_{X_{y_{0}}}))\neq 0$ for some $y_0 \in Y$, where $X_{y_0}=f^{-1}(y_0)$ and $\mathcal{O}(K_{X_y})$ stands for the canonical bundle over $X_y$.

Let us recall the relative (fibrewise) Bergman kernel $K_{X_y,L,h_L}(x)$ of $K_{X_y}+L|_{X_y}$ at each point $x$ belonging to each compact manifold $X_y$ defined by
\begin{equation}\label{chaoshi} 
\sup\{u(x)\otimes\overline{u}(x);u\in H^0(X_{y}, \mathcal{O}(K_{X_y}+L)\otimes \mathcal{I}(h_L |_{X_y})), \int_{X_y}|u|^2_{ h_L}\leq 1\}.
\end{equation}
Equivalently, the relative (fibrewise) Bergman kernel can be also expressed by
$$
\sum u_j \otimes \overline{u_j},
$$
where $\{u_j\}$ forms orthonormal basis of sections of $\mathcal{O}(K_{X_y}+L)\otimes \mathcal{I}(h_L |_{X_y})$ (see p.346 in \cite{bp08}). Note the Bergman kernel $K_{X_y,L,h_L}$ is not a function, but $\text{log}K_{X_y,L,h_L}$ induces a singular psh metric on $K_{X_y}+L|_{X_y}$ with semi-positive curvature current. By setting $\mathcal{O}(K_{X/Y})=\mathcal{O}(K_X-f^{\ast}K_Y)$, we can identify $\mathcal{O}(K_{X_y}+L|_{X_y})$ with $\mathcal{O}(K_{X/Y}+L)|_{X_y}$. Therefore, when $y$ varies in $Y$, The logarithm of the fibrewise Bergman kernel $\text{log}K_{X_y,L,h_L}$ induces a metric on  the line bundle  $\mathcal{O}(K_{X/Y}+L)$. 

The following well-known result states the plurisubharmonic variation of fibrewise Bergman kernels along the base directions, or equivalently, the plurisubharmonicity of the singular metric induced by fibrewise Bergman kernels.
\begin{theorem} [Theorem 1.3 in \cite{zhouzhujdg}, see also \cite{bp08} for the projective case]\label{domymath}
The logarithm of the fibrewise Bergman kernel (\ref{chaoshi}) of the line bundle $\mathcal{O}(K_{X/Y}+L)$ is plurisubharmonic on $X$.
\end{theorem}
\subsection{Approximation of Lelong numbers by complex singularity exponents}
In {\cite{guan-zhou1}}, the author has established a new approximation
process of plurisubharmonic functions.

\begin{theorem}[Theorem 1.3 in \cite{guan-zhou1}]
  \label{dd}Let $\varphi$ be a plurisubharmonic function on an open set $D
  \subset \mathbb{C}^n .$ Then for any $k \in \mathbb{N}$, there exists a
  plurisubharmonic function $\varphi_k$ defined on a neighborhood of $D \times
  \{ o \} \subset \mathbb{C}^n \times \mathbb{C}^k$ with coordinates $(z,
  \xi)$, such that
  \begin{enumerate}
    \item $\varphi_k (z, o) = \varphi (z)$,
    
    \item $\nu_{(z, o)} (\varphi_k) = \nu_z (\varphi)$,
    
    \item $\frac{k}{\nu_z (\varphi)} \leq c_{(z, o)} (\varphi_k) \leq \frac{n
    + k}{\nu_z (\varphi)}$
  \end{enumerate}
  for any $z \in D$, where $o$ is the origin in $\mathbb{C}^k$. One may take
  for instance
$$
    \varphi_k (z, \xi) = \sup_{\zeta \in B (z, | \xi |)  }
    \varphi (\zeta) .
$$
\end{theorem}

Their motivation is to express the upper-level set $E_c (\varphi)= \{ z
\in D ; \nu_z (\varphi) \geq c \}$ as
$$
  E_c (\varphi) = \bigcap_{k \in \mathbb{N}} \{ z \in D ; c_{(z, o)} ((n + k)
  \varphi_k) \leq c^{- 1} \}
$$
by 3) in Theorem \ref{dd} and consequently, give Siu's analyticity theorem another elegant proof. In our
case, a slight modification of the proof of Theorem \ref{dd} implies

\begin{lemma}
  \label{appr}Let $\varphi$ be a plurisubharmonic function defined on
  $\Delta^n_z \times \Delta^m_w$. Then there exists a plurisubharmonic
  function $\varphi_k$ defined on a neighborhood of $\Delta^n_z \times
  \Delta^m_w \times \{ o \} \subset \mathbb{C}^n \times \mathbb{C}^m \times
  \mathbb{C}^k$ with coordinates $(z, w, \xi)$, such that:
  \begin{enumerate}
    \item $\varphi_k (z, w, o) = \varphi (z, w)$,
    
    \item $\nu_{(z, o)} (\varphi_{k, w}) = \nu_z (\varphi_w)$,
    
    \item $\frac{k}{\nu_z (\varphi_w)} \leq c_{(z, o)} (\varphi_{k, w}) \leq
    \frac{n + k}{\nu_z (\varphi_w)}$
  \end{enumerate}
  for any $(z, w) \in \Delta^n_z \times \Delta^m_w$, where $o$ is the origin
  in $\mathbb{C}^k$ and $\varphi_{k, w_0} := \varphi_k |_{w = w_0} =
  \varphi_{k} (\cdot, w_0) $. If we further assume that $e^{\varphi}$ is
  H{\"o}lder continuous with respect to $w$, then for any $(z, w_1), (z, w_2)
  \in \Delta^n_z \times \Delta^m_w$,
  \begin{equation}
    e^{\varphi_k (z, w_1)} \leq e^{\varphi_k (z, w_2)} + C | w_1 - w_2
    |^{\alpha} \label{eqd} 
  \end{equation}
  holds for some uniform $C$ and $\alpha$. One may take for instance
  \begin{equation}
    \varphi_k (z, w, \xi) = \sup_{\zeta \in B (z, | \xi |) 
    } \varphi (\zeta, w), \label{o;}
  \end{equation}
  where $B (z, | \xi |)$ is the ball centered at $z$ with radius $|\xi|$ in $\mathbb{C}^n$.
\end{lemma}

\begin{proof}
  Since condition 1, 2, 3 has already been proved in Theorem \ref{dd}, our
  mission is to prove the second result (\ref{eqd}). Choose $\varphi_k$ as in
  (\ref{o;}), H{\"o}lder continuity of $e^{\varphi}$ yields that:
  \begin{equation}
    e^{\varphi  (\zeta, w_1)} \leq e^{\varphi  (\zeta, w_2)} + C | w_1 -
    w_2 |^{\alpha} \leq e^{\varphi_k (z, w_2)} + C | w_1 - w_2
    |^{\alpha} \label{sp}
  \end{equation}
  holds for any $\zeta \in B (z, | \xi |)  $. Once taking
  the supremum on the left hand side of (\ref{sp}), we get the desired result.
  Exchanging the position of $w_1$ and $w_2$, we will see that $e^{\varphi_k}$
  is still H{\"o}lder continuous.
\end{proof}

We remark that once Question \ref{xiaochou} has an affirmative answer, Theorem
\ref{se} will be a direct corollary from the following proposition and Lemma \ref{appr}.

From now on, we fix the notation on $\varphi_k$ as in (\ref{o;}).

\begin{proposition}\label{wanmei}
  \label{2.6}Let $\varphi$ be a plurisubharmonic function defined on
  $\Delta^n_z \times \Delta^m_w$. Then
$$
    X_c (\varphi) = \bigcap_{k \in \mathbb{N}} ( Y_{c^{- 1}} ((n + k)
    \varphi_k)  \bigcap \{ \xi = o \} ),
$$
where $Y_{c^{- 1}} ((n + k)
    \varphi_k)$ is interpreted as the following set
 \begin{equation}\label{interpret}
    \bigcup_{w\in\Delta^m}\{(z,w,\xi)\in\Delta^n\times\Delta^m\times\Delta^k;c_{(z,\xi)}(\varphi_{k,w})\leq c^{-1}, |\xi|<1-|z|\}.
  \end{equation}
\end{proposition}

\begin{proof}
  This result follows immediately from 3) in Lemma \ref{appr}.
\end{proof}

\subsection{Tools from birational geometry}

The main goal in this subsection is to recall some basic tools from birational
geometry. 

Let $X$ be a smooth projective variety, and
$\mathfrak{a} \subset \mathcal{O}_X$ be an ideal sheaf and $c > 0$. A
projective birational map $\mu : X' \rightarrow X$ is called a {{log
resolution}} of $\mathfrak{a}$ if
\begin{equation}\label{reso}
  \mu^{- 1} \mathfrak{a} = \mathfrak{a} \cdot O_{X'} = \mathcal{O }_{X'} (- F)
\end{equation}
for some effective divisor $F$ on $X'$ such that ${Exc} (\mu) + F$ has
simple normal crossing support. The existence of log resolution of $\mathfrak{a}$ is
guaranteed by Hironaka's resolution of singularities ({\cite{hironaka}}). Let us recall the analytic version of this important result.

\begin{theorem}[{{Hironaka's desingularization theorem, cf. main theorem in \cite{bl}}}]
  \label{hiro} Let $X$ be a
  complex manifold and $S$ be a(n) (analytic) subvariety in $X$. Then there exists a locally
  finite sequence of blow-ups $\mu_j : X_{j + 1} \rightarrow X_j$ $(j = 1, 2,
  \ldots)$ with smooth centers $Y_j$ such that:
  \begin{enumerate}
    \item Each component of $Y_j$ lies either in $(S_j)_{{sing}}$ or $S_j
    \bigcap E_j$, where $S_1 = S$, $S_{j + 1}$ denotes the strict transform
    of $S_j$ via $\mu_j$, $E_1 = \varnothing$ and $E_{j + 1}$ denotes the
    exceptional divisor $\mu^{- 1}_j (Y_j  \bigcup E_j)$;
    
    \item Let $S'$ and $E'$ denote the final strict transform of $S$ and the
    exceptional divisor respectively, then the underlying point set $|S' |$ is
    smooth and simultaneously $|S' |$ and $E'$ have only normal crossings.
  \end{enumerate}
\end{theorem}

\begin{definition}
  [{{Multiplier ideals associated to $\mathfrak{a}$ and $c$, Definition
  9.2.3 in {\cite{Larbook}}}}] Fix a log resolution $\mu : X' \rightarrow X$
  as in (\ref{reso}), the multiplier ideal associated to $\mathfrak{a}$ and $c$ is
  defined as
  \begin{equation}
    \mathcal{J} (X, \mathfrak{a}^c) := \mu_{\ast}\mathcal{O}_{X'} (K_{X' / X} - [c \cdot
    F]), \label{21155}
  \end{equation}
  where $[\cdot]$ means the round down process.
\end{definition}

It is clear that definition (\ref{21155}) does not depend on the choice of the
log resolutions (9.2.18 of {\cite{Larbook}}).

 Let $U\subset X$ be an affine coordinate and $\{f_j\}_{j=1}^{N}\in\mathbb{C}[X]$ be the generators of $\mathfrak{a}|_U$. Let $\varphi =
\varphi_{\mathfrak{a}}$ be the corresponding plurisubharmonic function defined as
$$
  \varphi =\varphi_{\mathfrak{a}}= \frac{1}{2} \cdot \log (\sum_{j = 1}^N |f_j |^2).
$$
Then it is easy to see
that $\mathcal{I} (c \cdot \varphi ) = \mathcal{J} (X, \mathfrak{a}^c)^{\textbf{An}}$, where $\cdot^{\textbf{An}}$ is the analytification process for (algebraic) coherent sheaves (cf. Theorem 9.3.42 in \cite{Larbook}). Conversely, the computation of $\mathcal{I}(\varphi)$ for $\varphi$ with analytic singularities can be reduced to a purely algebraic problem by using Hironaka's resolution of (analytic) singularities and the formula (\ref{21155}).

The following important theorem focuses on how (algebraic)
multiplier ideals vary in families.

\begin{theorem}
  \label{imp} {{(Theorem 9.5.35 and Example 9.5.37 in {\cite{Larbook}})}}
  Let $X$ and $T$ be two non-singular (irreducible, algebraic) varieties, $\mathfrak{a}
  \subset \mathcal{O}_X$ be an ideal, and $p : X \rightarrow T$ a smooth
  surjective mapping. Assume that 
  \begin{equation}\label{assss}
  \mathfrak{a}_t := \mathfrak{a} \cdot
  X_t\neq 0 \emph{ for \textbf{all} } t \in T,
  \end{equation}
  then
  \begin{equation}
    \mathcal{J} (X, \mathfrak{a }^c) |_{X_t} = \mathcal{J} (X_t, \mathfrak{a
    }_t^c) \label{220}
  \end{equation}
  holds for general $t \in T$ \emph{(i.e., outside some Zariski closed subsets of $T$)} and every $c > 0$.
\end{theorem}
\begin{remark}
  \label{rmk213} We want to show condition (\ref{assss}) could be replaced by the following assumption $(\ast)$ in the above theorem:
  $$(\ast)\quad \mathfrak{a} |_{X_t}\neq 0 \text{ for
  \textbf{general} } t \in T. $$

  In fact, let $\mu : X' \rightarrow X$ be the log resolution of
  $\mathfrak{a}$ and $\mu^{- 1} \mathfrak{a} = \mathcal{O}_{X'} (- F)$, where
  the support of $F$ denoted by $E_i$ has simple normal crossings. The
  assumption (\ref{assss}) is to assure
  that the prime divisor $E_i |_{X'_t}$ makes sense for all $t\in T$, where $X'_t : = (p \circ \mu)^{- 1} (t)$. However,
  if we substitute
  (\ref{assss}) by $(\ast)$, then we obtain $E_i |_{X'_t}\neq X'_t$ for
  general $t \in T$.
  Consequently, equality (\ref{220}) still holds for general $t \in T$ by
  the same proof (whose key point is applying log resolution of singularities and the theorems on local vanishing and generic smoothness, cf. Corollary 10.7 in \cite{hart}) as in {\cite{Larbook}}.
\end{remark}

\subsection{Restriction techniques}

Because of the lack of continuity, it may be difficult to generalize algebraic results to plurisubharmonic functions at times. For
instance, the fundamental adjunction exact sequence appeared in birational
geometry can be extended for plurisubharmonic functions with an additional
H{\"o}lder continuity assumption (cf. Theorem C in {\cite{guenancia2}}). For
general plurisubharmonic function, this result may not be true for the lack of
H{\"o}lder continuity of $e^{\varphi}$ (cf. {\cite{guanli}}).

Before stating the restriction techniques for H{\"o}lder continuous
plurisubharmonic functions, we recall the following useful lemma.

\begin{lemma}
  \label{nb}{{(Lemma 2.15 in {\cite{guenancia2}})}} Let $\Omega \subset
  \mathbb{C}^n$ be a domain and $\varphi$ a plurisubharmonic function defined
  on $\Omega$. Suppose that $\varphi \leq - 1$, $\alpha > 0$ and $f \in
  \mathcal{O} (\Omega)$. If
  \[ \int_{\Omega} \frac{| f |^2 e^{- 2 (1 + \varepsilon) \varphi} 
     }{(- \varphi)^{\alpha}} d \lambda < \infty \]
  for some $\varepsilon > 0$, then there exists some ${\varepsilon'}  > 0$
  such that
  \[ \int_{\Omega} | f |^2 e^{- 2 ( {1 + \varepsilon'}  ) \varphi}
     d \lambda < \infty .   \]
  
\end{lemma}

\begin{proof}
  Set $C := \inf \{ e^{\varepsilon x} / x^{\alpha} ; x \geq 1 \}$ and
  $\varepsilon' = \varepsilon / 2$, we know that:
  
$$
    \int_{\Omega} |  f | ^2 e^{- 2 ( {1 + \varepsilon'} ) \varphi} d \lambda \leq C^{- 1} \int_{\Omega} \frac{| f |^2 e^{- 2
    (1 + \varepsilon) \varphi}  }{(- \varphi)^{\alpha}} d
    \lambda < \infty . 
$$
  
  \ 
\end{proof}

Now we are ready to show the following key lemma, which was implicitly used in
the proof of Theorem 2.16 in {\cite{guenancia2}} to show the coherence of the
analytic adjoint ideal sheaves.

\begin{proposition}
  \label{rest}Let $\varphi$ be a plurisubharmonic function defined on
  $\Delta^n_z \times \Delta^m_w$. Suppose that $e^{\varphi}$ is H{\"o}lder
  continuous with respect to $w$, then for any
$$
    F \in {\mathcal{H}_{\Delta^n_z \times \Delta^m_w}}  ((1 + \varepsilon)
    \varphi + \log | w - w_0 |^{2 m} + \log (\log^2 | w - w_0 |)), 
$$
  $w_0 \in \Delta^m_w$ and $0 < r < 1$,
  \begin{equation}
    F |_{w = w_0} = F (\cdot, w_0) \in \mathcal{H}_{\Delta^n_z (r) \times
    \{ w_0 \}}( ( {1 + \varepsilon'}  ) \varphi)
    \label{claim} 
  \end{equation}
  for some ${\varepsilon'}  > 0$. Here $\Delta^n_z (r)$ denotes the polydisc
  with radius $(r, \ldots ., r) .$
\end{proposition}

\begin{proof}
  First find some $0 < r' < 1$ such that $w_0 \in \Delta^m_w (r')$. For any
  $(z, w_0), (z, w ) \in \Delta^n_z (r) \times \Delta^m_w (r')$, the
  mean-value theorem implies that
  \begin{equation}
    | F (z, w ) |^2 \geq | F (z, w_0) |^2 - C_1 | w -
    w_0 |^2 \label{gold} 
  \end{equation}
  for some uniform $C_1 > 0$ which only depends on $r$ and $r'$. Moreover, the
  H{\"o}lder continuity condition of $e^{\varphi}$ yields that
  \begin{equation}
    e^{2 \varphi (z, w )} \leq ( e^{\varphi (z, w_0)} + C_2 | w  -
    w_0 |^{\alpha}) \label{eqdd}^2 \leq C_3 (e^{2 \varphi (z, w_0)} + | w  -
    w_0 |^{2 \alpha})  
  \end{equation}
  for some uniform $C_2, C_3, \alpha > 0$. Setting $f (z) := F (z, w_0)$,
  we obtain the following inequality:
  \[ \frac{| F (z, w) |^2 e^{- 2 (1 + \varepsilon) \varphi (z, w)} 
     }{| w - w_0 |^{2 m} \log^2 | w - w_0 |  
      } \quad ( \text{by } ( \ref{eqdd} )
     ) \]
  \[ \geq C_3^{- (1 + \varepsilon)} \frac{| F (z, w) {| }^2 
     }{\log^2 | w - w_0 |  } \cdot \frac{1}{| w -
     w_0 |^{2 m} (e^{2 \varphi (z, w_0)} + | w  - w_0 |^{2 \alpha})^{(1 +
     \varepsilon)}    } \quad (
     \text{by } ( \ref{gold} ) ) \]
  \[ \geq \frac{C_3^{- (1 + \varepsilon)} | f (z) {| }^2 
     }{| w - w_0 |^{2 m} \log^2 | w - w_0 | (e^{2 \varphi (z, w_0)} + |
     w  - w_0 |^{2 \alpha})^{(1 + \varepsilon)}  
        } \]
  \begin{equation}
    - \frac{C_3^{- (1 + \varepsilon)} C_1^{- 1}}{| w - w_0 |^{2 m - 2} \log^2
    | w - w_0 | (e^{2 \varphi (z, w_0)} + | w  - w_0 |^{2 \alpha})^{(1 +
    \varepsilon)}     
    } . \label{1}
  \end{equation}
  Now letting $\rho (z) := \delta e^{\alpha^{- 1} \varphi (z, w_0)}$
  where $\delta > 0$ is chosen small enough to guarantee that $\{ | w - w_0 |
  < \rho (z) \}  $ is contained in $\Delta^n_z (r) \times
  \Delta^m_w (r')$, for the right term in the right hand side we have
  \[ \int_{| w - w_0 | < \rho (z)  } \frac{C_3^{- (1 +
     \varepsilon)} C_1^{- 1}}{| w - w_0 |^{2 m - 2} \log^2 | w - w_0 | (e^{2
     \varphi (z, w_0)} + | w  - w_0 |^{2 \alpha})^{(1 + \varepsilon)}
          } d
     \lambda_m \]
  \begin{equation}
    \leq e^{- 2 (1 + \varepsilon) \varphi (z, w_0)} C_4 \int_{| w - w_0 | <
    \rho (z)  } \frac{d \lambda_m}{| w - w_0 |^{2 m - 2}
     } \leq C_5 e^{2 [\alpha^{- 1} - (1 + \varepsilon)]
    \varphi (z, w_0)}, \label{2}
  \end{equation}
  where $C_3, C_4, C_5$>0 are some uniform constants. Therefore, the right
  hand side of (\ref{2}) has a uniform upper bound if $\alpha \leq (1 +
  \varepsilon)^{- 1}$. Note the first inequality uses an upper bound of
  $\log^{- 2} | w - w_0 |  $ as $| w - w_0 | \rightarrow 0
   $, and the second inequality follows from the spherical
  transformation of coordinates.
  
  As for the remaining term, we get
  \[ \int_{| w - w_0 | < \rho (z)  } \frac{C_3^{- (1 +
     \varepsilon)}}{| w - w_0 |^{2 m} \log^2 | w - w_0 | (e^{2 \varphi (z,
     w_0)} + | w  - w_0 |^{2 \alpha})^{(1 + \varepsilon)} 
         } d \lambda_{2 m}
  \]
  \[ \geq C_6 {e^{- 2 (1 + \varepsilon) \varphi (z, w_0)}}  \int_{| w - w_0 |
     < \rho (z)  } \frac{d \lambda_1}{| w - w_0 |^{2 m}
     \log^2 | w - w_0 |    } \]
  \[ \geq C_7 {e^{- 2 (1 + \varepsilon) \varphi (z, w_0)}}  \int_0^{\rho (z)}
     \frac{d r}{r (\log r)^2} \]
  \[ = - C_7 {e^{- 2 (1 + \varepsilon) \varphi (z, w_0)}}  \frac{1}{\log  \rho
     (z)} \]
  \begin{equation}
    = \frac{C_7 {e^{- 2 (1 + \varepsilon) \varphi (z, w_0)}} }{- \log \delta
    - \alpha^{- 1} \varphi (z, w_0)} . \label{3}
  \end{equation}
  Thanks to the estimate (\ref{1}), (\ref{2}), (\ref{3}) and the fact that
  \[ {\int_{\Delta^n_z \times \Delta^m_w}}  \frac{| F (z, w) |^2 e^{- 2 (1 +
     \varepsilon) \varphi (z, w)}  }{| w - w_0 |^{2 m}
     \log^2 | w - w_0 |    } d
     \lambda_{2 m + 2 n} < \infty, \]
  we infer that
$$
    \int_{\{ w_0 \} \times \Delta^n_z (r)} \frac{| f (z) |^2 {e^{- 2
    (1 + \varepsilon) \varphi (z, w_0)}}  }{- \varphi (z, w_0)}
    d \lambda_{2 n} < \infty
$$
  by Fubini's theorem. Now using Lemma \ref{nb}, our claim (\ref{claim})
  follows immediately.
\end{proof}

\subsection{Holomorphic stability of the integrals for 1-parameter
deformation}

The main purpose of this subsection is to prove the following result, which can also be viewed as a slight modification of Theorem 7 in
{\cite{phongsturm}}.

\begin{theorem}
  \label{useful}Let $D \subset \mathbb{C}^n$ be an open set containing $o$. Let $R (z, w)$ be
  a function of the form:
  \begin{equation}
    R (z, w) = R (z, w ; 2, \delta) : = \frac{| F(z,w) |^2  }{|
    G(z,w) |^{\delta}  } := \frac{\sum_{i = 1}^I | F_i (z,
    w) |^2  }{( \sum_{j = 1}^J | G_j (z, w)
    |^2)^{\delta / 2}  }, \label{222}
  \end{equation}
  where $F_i, G_j$ are both holomorphic functions on $D \times \Delta^1$ and
  $\delta$ is a fixed nonnegative rational number. Assume that
  
  \begin{equation}
    \int_D R (z, 0) d V (n) < \infty, \label{226}
  \end{equation}
  where $dV(n)$ is the Lebesgue measure of $\mathbb{C}^n$. Fix some $\beta \in \mathbb{Q}_+$ such that
  \begin{equation}
    \int_{{Q_1} } \frac{| F (z, 0) |^2  }{| G (z, 0)
    |^{\delta (1 + \beta)}  } < \infty, \label{227}
  \end{equation}
  where $Q_1$ is relative compact polydisc centered at $o$ in $D$. Now fix some $\alpha \in (1 -
  \beta, 1)$ and assume that
  \begin{equation}
    \int_{Q_2 \times \Delta (\varepsilon)} \frac{| F |^2 | G |^{- (1 + \beta)
    \delta}    }{| w |^{2 \alpha}
     } d V (n + 1) < \infty \label{219}
  \end{equation}
  for a smaller polydisc $Q_2 \subset Q_1 \subset D$ centered at $o$ and $\varepsilon > 0$.
  Then there is a small disc $\Delta^1 (\rho)$ for some $1 > \rho > 0$ and a relative open $D'
  \subset D$, such that:
$$
    w \mapsto \int_{D'} R (z, w) d V (n)
$$
  is both finite and continuous on $\Delta^1 (\rho)$.
\end{theorem}

\begin{remark}
  The original statement assumes plurisubharmonicity of $- \log R (z, c)$ to
  replace condition (\ref{219}). Their motivation is to apply Siu's lemma, or
  essentially, the $L^2$ extension theorem to get an estimate of the infimum
  of the integral (\ref{222}) parametrized by a puncture disk. Note
  (\ref{219}) is necessary since if we consider the following integration
$$
    I (w) : = \int_{\Delta} \frac{| z - w |  }{| z |^2
     } d V (1),
$$
  it is easy to see that $I (w)$ does not satisfies (\ref{219}) and
  simultaneously has not stable deformation around $w = 0$.
\end{remark}

\begin{remark}\label{spod}
  When $| F | \equiv \text{constant} > 0  $, such a
  holomorphic stability property still holds even if it is parametrized by
  multiple parameters. This result is implicitly proved in Theorem 4.2 of
  {\cite{guan-zhou-li}} by using an convergence theorem of weighted $L^2$ norm
  related to strong openness conjecture. 
\end{remark}

Before giving its proof, we highlight that there are two main ingredients for
the original proof of Theorem \ref{useful}. One is iterating sharp
estimates (Theorem~4 in {\cite{phongsturm}}) for one-dimensional integrals of
the form (\ref{222}), the other is the following Siu's lemma:

\begin{lemma}
  \label{siulemma}{{({\cite{phongsturm}}'s appendix)}} Let $D \subset
  \mathbb{C}^n$ be a pseudoconvex domain and $D' \subset D$ a relative compact
  subset. Let $\varphi (z, c)$ be a negative plurisubharmonic function on $D
  \times \Delta$ such that
$$
    \int_D e^{- \varphi (z, 0)} d V (n) < \infty,
$$
  then there exists a positive number $C > 0$ which only depends on $D$ and
  $D'$, such that
$$
    \inf_{0 < | c | < \rho  } \int_{D'} e^{- \varphi (z, c)}
    d V (n) \leq C \int_D e^{- \varphi (z, 0)} d V (n)
$$
  for $\rho > 0$ sufficiently small.
\end{lemma}

It is also worthwhile to mention that Zhou and Zhu gave a
generalized version of Lemma \ref{siulemma} (cf. \cite{zhouzhu2}) to establish an optimal $L^2$
extension theorem for singular weights on weakly pseudoconvex K\"{a}hler manifolds (cf. {\cite{zhouzhujdg}}). Their
proof of generalized version of Lemma \ref{siulemma} strongly relies on a solution to the strong openness conjecture of
multiplier ideal sheaves. Later on, they further developed Siu's
lemma for plurisubharmonic functions with non-trivial multiplier ideal
sheaves.

\begin{lemma}
  \label{zji}{{(Theorem 1.2 in {\cite{zhouzhu}}, $p = 2$ case)}} Let
  $\varphi (z, w)$ be a plurisubharmonic function defined on a product domain
  of unit balls $B^n_z \times B^m_w$, and $f (z)$ a holomorphic function on
  $B^n_w$ satisfying that:
$$
    \int_{B^n} | f (z) |^2 e^{- \varphi (z, 0)} d V (n) < \infty . 
$$
  Assume that $0 < r_1 < r_2 < r_3 < 1$, $\beta$ being a positive number such
  that
$$
    \int_{B^n (r_3)} | f (w) |^2 e^{- (1 + \beta) \varphi (z, 0)} d V (n)
   =: I_{\beta} < \infty  
$$
  {{(the existence of $\beta$ is guaranteed by
  {\cite{guan-zhou15soc}})}}, and $\alpha \in ( 1 - \frac{\beta}{m}, 1
  )$ be a non-negative number. Then there exists a holomorphic function
  $F (z, w)$ such that $F (z, 0) = f (z)$ on $B^n_z (r_3)$,
  \begin{equation}
    \int_{B^n_z (r_3) \times B^m_w} \frac{| F (z, w) |^2 
    }{| w |^{2 m \alpha}  } e^{- (1 + \beta)
    \varphi (z, w)} d V (m + n) < \infty \label{223},
  \end{equation}
  and
  \[ \lim_{\varepsilon \rightarrow 0} \frac{1}{\lambda_m (B ^m (\varepsilon))}
     \int_{B^n  (r_1) \times B^m_w (\varepsilon) } |F (z, w) |^2 e^{-
     \varphi (z, w)} d V (m + n) \]
  \begin{equation}
    = \int_{B^n (r_1)} |f (z)|^2 e^{- \varphi (z, 0)} d V (n) . \label{224}
  \end{equation}
  Moreover, any holomorphic extension $F$ of $f$ satisfying
  {{(\ref{223})}} must have {{(\ref{224}).}}
\end{lemma}

Now we are prepared to giving our proof of Theorem \ref{useful}, whose idea is
mainly based on the proof of Theorem 7 in {\cite{phongsturm}}. The proof will
be almost the same as {\cite{phongsturm}} except for showing the replacement
of assumption (\ref{219}), which has been essentially applied in step 2-step 4, does not affect our final destination. Note also the 
assumption (\ref{219}) behaves as stable as the original one in
{\cite{phongsturm}} (see step 1).

Let us first fix the notations
$$
  z' = (z_1, \ldots, z_{n - 1}), z'' = (z_1, \ldots, z_{n - 2}), \ldots
$$
for simplicity. 

Following \cite{phongsturm}, an {{absolute rational power}} in $n$ variables $Z =
(Z_{1,} \ldots, Z_n)$ is an expression of the following forms:
$$
  R = R (Z ; \varepsilon, \delta) = \frac{(\sum_{i = 1}^I |F_i (Z)
  |^2)^{\varepsilon / 2}}{(\sum_{j = 1}^J |G_j (Z) |^2)^{\delta / 2}} = \frac{K
  (Z)}{L (Z)},
$$
where $F_i, G_j$ are polynomials in $Z$ with coefficients in a ring $A$, $\varepsilon$ and $\delta$ are nonnegative rational numbers, and at least one of
the $F_i$'s is nonzero. All
absolute rational powers with coefficients in $A$ will be denoted by the set
$A |Z|$. When $F_i, G_j$ are Weierstrass polynomials in the $z_n$ directions, a
pair of exponents $(\varepsilon, \delta) \subset \mathbb{Q }^2_{> 0}$ of $R
(z, w ; \varepsilon, \delta)$ is called \emph{non-degenerate} if
\begin{equation}
  \nu \varepsilon + 2 - \frac{l}{[\delta] + 1} \delta \neq 0 \label{239}
\end{equation}
for any integers $0 \leq \nu \leq \max_{i\in I} M_i$ and $ 0 \leq l \leq \max_{j\in J} N_j$, where
$$
  |F|^2 = \sum_{i = 1}^I |F_i |^2 = \sum_{i = 1}^I | \sum_{\mu = 0}^{M_i}
  a_{i, \mu} (z' ,w) z_n^{\mu} |^2
$$
and
$$
  |G|^2 = \sum_{j = 1}^J |G_j |^2 = \sum_{j = 1}^J | \sum_{\nu = 0}^{N_j}
  b_{j, \nu} (z' ,w) z_n^{\nu} |^2.
$$
\begin{proof}
  [Proof of Theorem \ref{useful}] Thanks to Lemma 5.4 in \cite{phongsturm}, we just need to show
  \begin{equation}\label{goal}
    \sup_{w \in \Delta (\rho)} \int_{P } R (z, w ; 2, \delta + \eta) d V (n) <
    \infty
  \end{equation}
  for some $\rho, \eta > 0$, where $P$ is a small polydisc in $\mathbb{C}^n$. To make things clear, we divide the proof into
  four steps:
  \begin{step}
   Reduction of each $F_i,G_j$ into
  Weierstrass polynomials with non-degenerate exponents $(2,\delta)$
  \end{step}
  
  Thanks to Lemma 5.3 in {\cite{phongsturm}}, we may assume that each $F_i (z,
  w), G_j (z, w)$ are Weierstrass polynomials in $z_n$ in a smaller polydisc
  $B^{n - 1}_{r_{n - 1}} \times B^1_{r_n} \times B^1_{\rho} \subset \subset
  Q_2$ for some $r_{n - 1}, r_n, \rho > 0$ after a linear change of $z$
  variables.
  
  Let us choose $\eta > 0$ such that (\ref{226}) still holds for $R (z, w ;
  2, \delta + \eta)$ (after possibly shrinking $D$) and the pair $(2, \delta + \eta)$ is non-degenerate. This
  is always possible by an openness argument and the fact that $(2, \delta +
  \eta)$ is degenerate (i.e., (\ref{239}) does not hold) for a finite choice
  of numbers $\eta$. We can also arrange that (\ref{219}) still
  holds for this choice of $\eta$ (after possibly shrinking $Q_2$ and $\varepsilon$).
  
  In summary, we may now assume each $F_i (z, w), G_j (z, w)$ are Weierstrass
  polynomials and the pair $(2, \delta)$ is non-degenerate.
  \begin{step}
 Lower the integral domain into a
  $n - 1$ dimensional polydisc
  \end{step}
  Thanks to our non-degenerate assumption and Theorem 4 in
  {\cite{phongsturm}}, we obtain that ($d V_i$ will stand for the 1-dimensional volume form along the $z_i$-direction)
  \begin{equation}\label{sv}
    \int_{B_{r_n}^1} R (z',z_n, w) d V_n \sim T_{\lambda (a (z', w), b (z', w))}
    ((a (z', w), b (z', w))) := T_{\lambda} (z', w),
  \end{equation}
  where $\lambda(a (z', w), b (z', w)):=\lambda (z', w) \in \{0, \ldots, \mathcal{N}\}$ denotes the
  index (which is determined by $a,b$) for a finite chain of algebraic varieties
$$
    \{(a, b) = (a_{i, \mu}, b_{j, \nu}) |a_{i, \mu}, b_{j, \nu} \in
    \mathbb{C}, a_{i, \mu} \neq 0 \text{ for some } i, \mu\} = : U_0
    \supset \ldots \supset U_{\mathcal{N}} = \varnothing
$$
  and $T_{\lambda} \in \mathbb{Z} | (a , b ) |$. The implied constant in (\ref{sv}) is independent of $(z',w)$. Let $\lambda' := \inf_{(z',
  w) \in B^{n - 1}_{r_{n - 1}} \times B^1_{\rho}} \lambda (z', w)$, and write
  $T_{\lambda'}$ as
$$
    T_{\lambda'}(z',w) = \frac{K_{\lambda'} (z', w)}{L_{\lambda'} (z', w)}
$$
  for some absolute polynomials $K_{\lambda'}, L_{\lambda'}$.
  
  We claim that{ {$L_{\lambda'}  \not\equiv 0$}
  {(for the sake of (\ref{248}-\ref{250}))}}. Otherwise, the
  integral (\ref{226}) would be infinite for all $w$ belonging to a small
  neighborhood of $w = 0$. Then Lemma \ref{zji} implies that
$$
    \lim_{\varepsilon \rightarrow 0} \frac{1}{\pi \varepsilon^2} \int_{P 
    \times \Delta (\varepsilon)} R (z, w) d V (n+1)  = \int_{P } R (z, 0) d
    V (n) < \infty
$$
  for any $P \subset \subset Q_2$ {{by our
  assumption (\ref{219})}}. This is clearly impossible since the left hand
  side is infinite if $L_{\lambda'} \equiv 0$.
  
  Applying Theorem 4 in {\cite{phongsturm}} again, it implies that, for each nonzero
  $w \in \Delta (\rho)$ and all $z'$ outside the subvariety $Z'_w$ defined by
  $Z_w' = \{z' : K_{\lambda'} (z', w) = 0\},$  
  \begin{equation}
    \int_{B_{r_n}^1} R (z, w) d V_n \sim T_{\lambda'} (z', w) = |w|^{\mu_1}
    \cdot \tilde{T}_{\lambda'} (z', w) \label{247}
  \end{equation}
  holds. Here $\mu_1 \in \mathbb{Q}$, $\tilde{T}_{\lambda'} =
  (\tilde{K}_1)^{\varepsilon_1} / (\tilde{L_1})^{\delta_1}$, $\tilde{K}_1$
  and $\tilde{L}_1$ are sums of absolute values of holomorphic functions, and
  $\tilde{K}_1 (z', 0)  \not\equiv 0$, $\tilde{L}_1 (z', 0)  \not\equiv 0$. Note
  $\tilde{K}_1 (z', 0)  \not\equiv 0$ implies that $Z_w'$ is a proper
  subvariety{ (in particular of measure zero, for the sake of
  the following multiple integrals, e.g. \ref{forj})} for all nonzero $w \in B^1_{\rho}$ after shrinking $\rho$ a little bit.
  \begin{step}
 An iteration process of the first
  two steps
  \end{step}
  After making a linear change in the $z'$ variables and using Lemma 5.3 in
  {\cite{phongsturm}} again, we may assume that $\tilde{K}_1 (z', w),
  \tilde{L}_1 (z', w)$ are sums of absolute values of monic polynomials in
  $z_{n - 1}$ whose coefficients are analytic functions of $(z'', w)$ for $z'=(z'',z_n,w)
  \in B_{r_{n - 2}}^{n - 2}\times B^1_{r_{n-1}} \times B^1_{\rho}$ (possibly after shrinking $\rho$ and $r_{n-1}$) and some $r_{n - 2} > 0$. Now choose $\sigma_1 >
  0$ small enough such that:
  \begin{enumerate}
    \item \label{aa} $(\tilde{K}_1)^{\varepsilon_1} / (\tilde{L'_1})^{\delta_1 +
    \sigma_1}$ is non-degenerate with respect to the $z_{n - 1}$ variable;
    
    \item \label{bb}
    \begin{equation}
      \int_{B^{n - 2}_{r_{n - 2}} \times B^1_{r_{n - 1}} \times B^1_{r_n}}
      \frac{|F (z, 0) |^2 | G (z, 0) |^{- \delta} 
      }{\tilde{L}_1 (z', 0)^{\sigma_1}} d V (n) < \infty \label{248}
    \end{equation}
    and
    \begin{equation}
      \int_{B^{n - 2}_{r_{n - 2}} \times B^1_{r_{n - 1}} \times B^1_{r_n}}
      \frac{|F (z, 0) |^2 | G (z, 0) |^{- (1 + \beta) \delta} 
      }{\tilde{L}_1 (z', 0)^{(1 + \beta) \sigma_1}} d V (n) < \infty
      ; \label{249}
    \end{equation}
    {(
     we might shrink $r_{n-2},r_{n-1},r_n$ again for the use of strong openness property)}
    
    \item \label{cc}
    \begin{equation}
      \int_{S_1 \times \Delta (\varepsilon_1)} \frac{| F (z, w) |^2 | G (z, w)
      |^{- (1 + \beta) \delta}    }{|
      w |^{2 \alpha} \tilde{L}_1 (z', w)^{(1 + \beta) \sigma_1} 
      } d V (n + 1) < \infty \label{250}
    \end{equation}
    holds with $B^{n -
    2}_{r_{n - 2}} \times B^1_{r_{n - 1}} \times B^1_{r_n} \subset \subset S_1
    \subset \subset Q_2, \varepsilon_1 < \varepsilon$.
  \end{enumerate}
  This can be done as the first step as before. Applying
  (\ref{248})-(\ref{250}) and{ Lemma \ref{zji}} again, we
  obtain
  \begin{equation}
    \lim_{\varepsilon \rightarrow 0} \frac{1}{\pi \varepsilon^2} \int_{{P_2} 
    \times \Delta (\varepsilon)} \frac{R (z, w)}{\tilde{L}_1 (z',
    w)^{\sigma_1}} d V (n+1)  = \int_{{P_2} } \frac{R (z, 0)}{\tilde{L}_1
    (z', 0)^{\sigma_1}} d V (n) < \infty \label{240} ,
  \end{equation}
  where $P_2=B^{n -2}_{r_{n - 2}} \times B^1_{r_{n - 1}} \times B^1_{r_n}$.
  
  Now we get for all nonzero $w \in B^1_{\rho}$,
  \begin{equation}\label{forj}
    \int_{B^1_{r_{n - 1}}} \int_{B^1_{r_n}} \frac{R (z, w)}{\tilde{L}_1 (z',
    w)^{\sigma_1}} d V_n d V_{n - 1} \sim |w|^{\mu_1} \int_{B^1_{r_{n - 1}}}
    \frac{(\tilde{K_1} (z', w))^{\varepsilon_1}}{(\tilde{L}_1 (z',
    w))^{\delta_1 + \sigma_1}} d V_{n - 1}
  \end{equation}
  holds thanks to (\ref{247}). We may apply Theorem 4 in {\cite{phongsturm}} again
  by condition (\ref{aa}) and find that
$$
   |w|^{\mu_1} \int_{B^1_{r_{n - 1}}} \frac{(\tilde{K}_1 (z',
    w))^{\varepsilon_1}}{(\tilde{L}_1 (z', w))^{\delta_1 + \sigma_1}} d V_{n -
    1} \sim |w|^{\mu_1} T_{\lambda''} (z'', w) := |w|^{\mu_2} \tilde{T}_{\lambda''}
    (z'', w)
$$
  for nonzero $w\in B^1_{\rho}$ (possibly after shrinking) and all $z'' \in B^{n - 2}_{r_{n - 2}}$ outside
  a proper subvariety $Z''_w$ which depends on $w$. 
  
  If we write $\tilde{T}_{\lambda''}
  = (\tilde{K}_2)^{\varepsilon_2} / (\tilde{L}_2)^{\delta_2}$ and assume that $\tilde{K}_2 (z'', w),
  \tilde{L}_2 (z'', w)$ are sums of Weierstrass polynomials in $z_{n-3}$ for $z''=(z''',z_{n-2},w)\in B^{n-3}_{r_{n-3}}\times B^1_{r_{n-2}} \times B^1(\rho)$ and some $r_{n-3}>0$ as before, then
  we can choose $\sigma_2 > 0$ such that
  \begin{enumerate}
    \item $(\tilde{K}_2)^{\varepsilon_2} / (\tilde{L}_2)^{\delta_2 +
    \sigma_2}$ is non-degenerate with respect to the $z_{n - 2}$ variable;
    
    \item
$$
     \int_{B^{n - 3}_{r_{n - 3}} \times B^1_{r_{n - 2}} \times B^1_{r_{n -
       1}} \times B^1_{r_n}} \frac{|F (z, 0) |^2 | G (z, 0) |^{- \delta}
        }{\tilde{L}_1 (z', 0)^{\sigma_1} \tilde{L}_2 (z'',
       0)^{\sigma_2}} d V (n) < \infty 
$$
    and
$$
      \int_{B^{n - 3}_{r_{n - 3}} \times B^1_{r_{n - 2}} \times B^1_{r_{n -
      1}} \times B^1_{r_n}} \frac{|F (z, 0) |^2 | G (z, 0) |^{- (1 + \beta)
      \delta}  }{\tilde{L}_1 (z', 0)^{(1 + \beta)
      \sigma_1} \tilde{L}_2 (z'', 0)^{(1 + \beta) \sigma_2}} d V (n) < \infty
      ;
$$
    \item
$$
      \int_{S_2 \times \Delta (\varepsilon_2)} \frac{| F (z, w) |^2 | G (z, w)
      |^{- (1 + \beta) \delta}    }{|
      w |^{2 \alpha} \tilde{L}_1 (z', w)^{(1 + \beta) \sigma_1} \tilde{L}_2
      (z'', w)^{(1 + \beta) \sigma_2}  } d V (n + 1) <
      \infty
$$
    holds 
    with $B^{n -
    3}_{r_{n - 3}} \times B^1_{r_{n - 2}} \times B^1_{r_{n - 1}} \times
    B^1_{r_n} \subset \subset S_2 \subset\subset S_1, \varepsilon_2 < \varepsilon_1<\varepsilon$.
  \end{enumerate}
  ...
 \begin{step}
 Application of generalized Siu's
  lemma (i.e. Lemma \ref{zji})
  \end{step}
  
  In summary, by the iteration and induction process as above, we finally obtain that for each nonzero $w\in\Delta(\rho)$,
  \begin{equation}
    \int_{P_n} \frac{R (z, w)}{\Lambda (z, w)} d V (n) \sim T (w) \label{2237}
    \sim| w |^{\mu_n},
  \end{equation}
  where the left hand side is finite at $w = 0$, $\Lambda (z, w) = \prod_{k =
  1}^n \tilde{L}_k  (z, w)^{\sigma_k}$ , $P_n = \prod_{k = 1}^n B^1_{r_k}$.
  Here for each $1 \leq j \leq n$, we require that each $\tilde{L}_j$ and
  $\sigma_j > 0$ not only satisfy that $(\tilde{K}_j)^{\varepsilon_j} /
  (\tilde{L}_j)^{\delta_j + \sigma_j}$ are non-degenerate with respect to
  $z_{n - j}$, but also satisfy the following integrable conditions:
  \begin{equation}
    \int_{P_j} \frac{| F |^2  }{| G |^{\delta}
    \cdot \prod_{k = 1}^j L_k (z_1, \ldots, z_{n - k}, 0)^{\sigma_k} } dV(n)< \infty, \label{257}
  \end{equation}
  \begin{equation}
    \int_{P_j} \frac{| F |^2  }{| G |^{(1 +
    \beta) \delta} \cdot \prod_{k = 1}^j L_k (z_1, \ldots, z_{n - k}, 0)^{(1 +
    \beta) \sigma_k} }dV(n) < \infty, \label{j}
  \end{equation}
  
  \begin{equation}
    \int_{S_j \times \Delta (\varepsilon_j)} \frac{| F |^2 | G |^{- (1 +
    \beta) \delta}    }{| w
    |^{2 \alpha} \prod_{k = 1}^j L_k^{(1 + \beta) \sigma_k} } d V
    (n + 1) < \infty \label{238}
  \end{equation}
  for $\varepsilon > \varepsilon_1 > \cdots > \varepsilon_j > 0$, and $S_j
  \subset \ldots . \subset S_1 \subset Q_2$, and $P_j=B^{n-j}_{r_{n-j}}\times\prod_{k=1}^j B^1_{r_{n-j+1}}\subset\subset S_j$. As a result of (\ref{257}),
  (\ref{j}) and (\ref{238}) with $j = n$ and applying Lemma \ref{zji} with
  $\varphi = \delta \log | G | + \log \Lambda (z, w)  $,
  one concludes finally that
  \begin{equation}
    \lim_{\varepsilon \rightarrow 0} \frac{1}{\pi \varepsilon^2} \int_{{P_n} 
    \times \Delta (\varepsilon)} \frac{R (z, w)}{\Lambda (z, w)} d V (n+1) 
    = \int_{{{P_n} } } \frac{R (z, 0)}{\Lambda (z, 0)} d V (n) < \infty
    \label{2409}
  \end{equation}
  for ${P_n}  \subset \subset S_n$. Equality (\ref{2409}) implies that the left hand
  side of (\ref{2237}), as a function of $w$, is bounded with respect to some subsequence $\{ w_k \}_{k=1}^\infty$ approaching
  to zero. In particular, we obtain from (\ref{2237}) that
$$
    \sup_{w \in \Delta (\rho)} \int_{{P_n} } \frac{R (z, w)}{\Lambda (z, w)} d
    V (n) < \infty,
$$
  and consequently
$$
    \sup_{w \in \Delta (\rho)} \int_{P_n} R (z, w) d V (n) < \infty ,
$$
     which actually confirms (\ref{goal}).
\end{proof}

\section{Proof of main results and applications}\label{sec3}

In this section, we will first give proofs of Theorem \ref{polar} and Theorem \ref{third}.

\begin{proof}
  [Proof of Theorem \ref{polar}] Thanks to
  (\ref{2626}) (which is proved by the same argument as in Proposition \ref{213}) and the definition of relative
  Bergman kernels (\ref{28}), we obtain
  \begin{equation}
    Y_c (\varphi) = \{ (z, w) \in \Delta^n \times \Delta^m ; \log K_{c
    \varphi} (z, w) = - \infty \}. \label{3.1}
  \end{equation}
  Then by using Theorem \ref{2.4} we infer that $Y_c (\varphi)$ is a complete pluripolar set for any $c > 0$. In fact, equality (\ref{3.1}) was already implicitly
  proved and used in {\cite{guancrelle}}.
  
  In order to verify that $X_{c}(\varphi)$ is a pluripolar set, according to Proposition~\ref{2.6}, we just need to show that for some $k \in \mathbb{N}$, $$Y_{c^{- 1}} ((n + k) \varphi_k) 
  \bigcap \{ \xi = o \} = \{(z,w,o_\xi);\log K_{c^{- 1} (n + k) \varphi_k} (z, w, o_{\xi})
  = - \infty\}$$ (see (\ref{interpret})) is a proper (and therefore complete pluripolar by the fact that (\ref{3.1}) is complete pluripolar) subset of $\Delta^n \times \Delta^m$. This is so
  since otherwise 
 $\varphi(z,w)=\varphi_k (z,w,o_\xi) \equiv - \infty$ on $\Delta^n \times \Delta^m$. 
\end{proof}

\begin{proof}
  [Proof of Theorem \ref{third}] Let us first assume that $c = 1$
  without loss of generality. Now consider the following two statements for
  any $m \in \mathbb{N}$:
  
  {\noindent}{{Statement $A_m$.}} $Z_1 (\varphi)$ is an
  analytic subset for any quasi-psh $\varphi \in \mathcal{} {QPSH} (X)$ with analytic singularities and compact complex manifold $Y$, where $f:X\rightarrow Y$ is a surjective holomorphic submersion and $Y$ is of dimension $m$.{\medskip}
  
  {\noindent}{{Statement $B_m$.}} $Z_{1, S} (\varphi)
  := \bigcup_{s \in S} \{z \in X_s; c_z (\varphi | _{X_s}) \leq 1\}$ is an
  analytic subset for any quasi-psh $\varphi \in {QPSH} (X) $ with analytic
  singularities and any (irreducible) projective subvarieties $X$ and $S$, where $f:X\rightarrow S$ (which is precisely induced by a surjective submersion between the ambient manifolds, i.e. there exist ambient projective manifolds $X_{\text{amb}}\supset X$, $S_{\text{amb}} \supset S$ and a holomorphic surjective submersion $f_{\text{amb}}: X_{\text{amb}}\rightarrow S_{\text{amb}}$ so that $f = f_{\text{amb}}|_{X}$) is surjective with smooth fibers and $S$ is of pure dimension $m$. Here $\varphi \in {QPSH} (X) $ has analytic singularities precisely means that $\varphi$ is the restriction of a quasi-psh function with analytic singularities defined on the ambient smooth manifold $X_{\text{amb}}$.{\medskip}
  
  Let us first prove that $A_m$ implies $B_m$.
  
  Actually, we first take a log resolution $\mu :
  \tilde{S} \rightarrow S$ such that $\tilde{S}$ is smooth due to
  Theorem \ref{hiro}. Then the fiber product $X\times_{S} \tilde{S}$ becomes a smooth manifold (since it is the preimage of a submanifold $\tilde{S}$ by an ambient surjective submersion) and the induced map $f':=X\times_{S} \tilde{S}\rightarrow\tilde{S}$ is also a surjective submersion. Let $\mu':= X\times_{S} \tilde{S}\rightarrow X$ be the induced proper mapping. It is clear that the quasi-psh function $\mu'^{\ast} \varphi$ has analytic
  singularities as well on $X\times_{S} \tilde{S}$ by the construction of $\mu$ in Theorem
  \ref{hiro}. Note $\mu'$ also does not change the value of
  $\varphi$ along each fiber $X_s$ (i.e. $\varphi(p) = \mu'^{*}\varphi(p') $ holds for any $s\in S$, $p\in f^{-1}(s)$ and $p' \in \mu'^{-1}(p)$), we obtain that $\mu' (Z_1 (\mu'^{\ast} \varphi))
  = Z_{1,S} (\varphi)$, where $Z_1 (\mu'^{\ast} \varphi)=\bigcup_{s'\in S'}\{p\in (f')^{-1}(s');c_p (\mu'^{\ast} \varphi)\leq 1\}$. Thanks to the assumption $A_m$ we obtain $Z_1 (\mu'^{\ast} \varphi)$ is an analytic subset. Then by Remmert's proper
  mapping theorem, it follows that $Z_{1,S} (\varphi)$ is also an analytic
  subset in $X$.
  
  Next we claim that $B_0, \ldots, B_m$ implies $A_{m + 1}$.
  
As mentioned in Remark \ref{jiaowai1}, we already see that $Y$ is a $m+1$ dimensional projective manifold and $f$ corresponds to a morphism by GAGA. By the fact that $\varphi$ has analytic singularities, there is an algebraic integrally closed ideal sheaf $\mathfrak{a}\subset\mathcal{O}_{X}$ corresponding to $\varphi$ so that $\mathcal{J}(X,\mathfrak{a}^c)^\textbf{An}=\mathcal{I}(\varphi)$ for some $c>0$ and $\mathcal{J}(X,\mathfrak{a}|_{X_y}^c)^\textbf{An}=\mathcal{I}(\varphi|_{X_{y}})$ holds for any $y\in Y$. Hence according to Theorem \ref{imp} and its remark, we obtain that
  \begin{equation}
    \mathcal{I} (\varphi |_{X_y})
    = \mathcal{I} (\varphi) \cdot \mathcal{O }_{X_y} \label{41}
  \end{equation}
  holds for general $y$ outside some analytic subset $W \subset Y$.
  We might assume that $W$ consists of irreducible components $W =
  \bigcup_{i = 1}^N W_i$, and $\dim W_i \leq m$.
  
  Now we need to prove
  \begin{equation}\label{lab}
    Z_1 (\varphi) = F_1 (\varphi)  \bigcup (\bigcup_{i = 1}^N Z_{1, W_i}
    (\varphi)),
  \end{equation}
  which yields the analyticity of $Z_1 (\varphi)$ by our assumption of $B_0,
  \ldots, B_m$ and the analyticity of $F_1 (\varphi)$. Here $F_1 (\varphi):=\{x\in X;c_x (\varphi)\leq 1\}$ and $Z_{1, W_i} (\varphi ) = \bigcup_{w_i \in W_i} \{p \in
  X_{w_i}; c_p (\varphi_{w_i}) \leq 1\}$.
  
  In fact,
  \[ Z_1 (\varphi) = \bigcup_{y \in Y} \{p\in X_y ; c_p (\varphi |_{X_y})
     \leq 1\} = \bigcup_{y\in Y} \{p\in X_y ; \mathcal{I} 
     (\varphi |_{X_y})_p \neq \mathcal{O}_{X_y ,p} \} \]
  \begin{equation}
    =  \{x\in X\backslash f^{-1}(W) ; \mathcal{I}
    (\varphi )_{x} \neq \mathcal{O}_{X,x} \}  \bigcup (\bigcup_{i =
    1}^N Z_{1, W_i} (\varphi)) \label{43}
  \end{equation}
  \begin{equation}
    = F_1 (\varphi)  \bigcup (\bigcup_{i = 1}^N Z_{1, W_i} (\varphi)),
    \label{44}
  \end{equation}
  where the third identity (\ref{43}) is obtained from (\ref{41}) and the
  Ohsawa-Takegoshi $L^2$ extension theorem, and the last identity (\ref{44})
  has applied the restriction formula (cf. Theorem \ref{resc}) or the Ohsawa-Takegoshi $L^2$ extension
  theorem again. Hence (\ref{lab}) has been proved.
  
  Therefore, analyticity of $Z_1 (\varphi)$ in any dimension follows from the easy fact that $B_0 = A_0$
  is an analytic subset and the induction statement $A_m , B_m$ as above.
  
 For the "moreover" part, by Remmert's mapping theorem and the fact that $f(Z_{1}(\varphi))\neq Y$, it follows easily that the last identity in (\ref{sulia}) holds for general $y\in Y\backslash f(Z_{1}(\varphi))$. Thanks to restriction formula of multiplier ideals, the whole equality in (\ref{sulia}) holds for $y\in Y\backslash f(Z_{1}(\varphi))$.
\end{proof}

Among the proof of Theorem \ref{third}, we are essentially using the algebro-geometric log
resolution of singularities due to Hironaka. By contrast, our next
task is to show the proof of Theorem \ref{mainresult} via an analytic
approach without using resolution of singularities. Another advantage is that
it exactly shows which analytic functions in $\mathcal{O}_{n + 1}$ defining each germs of $Y_c (\varphi)$.

\begin{proof}
 [Proof of Theorem \ref{mainresult}] By scaling,
  we may assume that $c = 1$. Since our claim is local, we might also assume
  that $\varphi = c_0 \log ( \sum_{j = 1}^N | G_j |^2) 
  $ for some $c_0 > 0$, $G_j \in \mathcal{O} (\Delta^n_z \times \Delta
  _w)$ and $e^{\varphi}$ is H{\"o}lder continuous on $\Delta^n_z \times \Delta
  _w$. Our mission is to show the analyticity of the germ of $(Y_1 (\varphi),
  o = (o_z, o_w)) .$
  
  According to Proposition \ref{213}, we can write
  
$$
    Y_1 (\varphi) = \bigcup_{w \in \Delta} ( \bigcap_{f_w \in \bigcup_{c
    > 1 } \mathcal{H}_{\Delta^n \times \{ w \}} (c \varphi_w)} f^{- 1}_{{{ }
    }_w} (0) )=: I_1 .
$$
  
  We next define an analytic subset of $\Delta^n_z \times \Delta _w$ by
  \begin{equation}
    I_2 := \bigcap_{F \in \bigcup_{c > 1} \mathcal{H}_{\Delta^n \times
    \Delta } (c \varphi + \log | w |  + \log (- \log  | w  |))_{\mathcal{}}
       } F^{- 1} (0) \label{calim}
  \end{equation}
  and claim that
$$
    (I_1, o) = (I_2, o),
$$
  by which analyticity of $Y_1 (\varphi)$ at $o$ in $\Delta^n \times \Delta $
  follows immediately.
  
   First note that for any $( z  {, w_0}  ) \in
  I_2$, $c > 1$ and $f_{w_0} \in \mathcal{H}_{\Delta^n \times \{ w_0 \}} (c
  \varphi_{w_0})$, Lemma \ref{ote} yields that if $w_0 \neq 0$, then there
  exists
$$
    F \in \mathcal{H}_{\Delta^n \times \Delta} (c \varphi + \log | w |)
    \subset \mathcal{H}_{\Delta^n \times \Delta } (c \varphi + \log | w | +
    \log (- \log  | w |)),    
$$
  such that $F (\cdot, w_0) = f_{w_0}$; if $w_0 = 0$, then there exists
$$
    F \in \mathcal{H}_{\Delta^n \times \Delta } (c \varphi + \log | w | + \log
    (- \log  | w |))    
$$
  such that $F (\cdot, w_0) = f_{w_0}$. In any case, we get that $f_{w_0} (z)
  = F (z, w_0) = 0$ and $(I_2, o) \subset (I_1, o) .$
  
   Conversely, thanks to
  the Noetherian property of the coherent sheaf $\mathcal{O}_{m + n}$, one may
  find a small polydisc $U$ such that
$$
    I_2  \bigcap U = \{ (z, w) \in U ; F_j (z, w) = 0, j = 1, \ldots ., M \} .
$$
  By Proposition \ref{rest}, it implies that for any $F_j \in
  \mathcal{H}_{\Delta^n \times \Delta } (c_j \varphi + \log | w |  + \log (-
  \log | w  |))    $ for some $c_j > 1$,
$$
    F_j (\cdot, o_w) \in \mathcal{H}_{\Delta^n (r) \times \{ o_w \}} (c'_j
    \varphi_{o_w})
$$
  holds for some $c'_j > 1$ and $0 < r < 1$. Now the holomorphic stability Theorem
  \ref{useful} can be applied for $F := (F_1, \ldots ., F_M)$ since
  (\ref{219}) and (\ref{226}) are both valid for the weight $e^{- 2 c''
  \varphi } = (| G_1 |^2 + \cdots + | G_N |^2)^{- c''}  
   $ with some $1 < c'' \in \mathbb{Q} < c_{\min} :=
  \min \{ c'_j \}$ and $1 + \beta = c_{\min} / c''>1$. As a consequence, one may
  find a smaller relative compact polydisc $U' = \Delta^n (r_1) \times
  \Delta  (r_2)$ such that
$$
    w \mapsto \int_{\Delta^n (r_1)} \frac{\sum_{j = 1}^M | F_j (z, w) |^2
     }{e^{2 c'' \varphi}} d V (z)
$$
  is both continuous and finite when $w \in \Delta (r_2)$. Hence the
  inclusion $I_1 \bigcap U' \subset I_2 \bigcap U'$ holds.
\end{proof}

\begin{remark}\label{jinaj}
  Indeed, what we have proved is that $(Y_c (\varphi), (z_0, w_0))$ is nothing
  but the germ of the supporting analytic subset of the coherent analytic sheaf $\text{Supp}
  \mathcal{O}_{n + m} / {Adj}_{\{w = w_0 \}} (c \varphi)$ at $(z_0, w_0)$,
  where ${Adj}_{\{w = w_0 \}} (\varphi)$ is the analytic adjoint ideal
  sheaf (see Definition 2.1 in {\cite{guenancia2}}) with respect to the
  hyperplane $\{w = w_0 \} .$
\end{remark}

We can now eventually present the proof of Corollary \ref{cor1.6} and Corollary~\ref{second}.

\begin{proof}
  [Proof of Corollary \ref{cor1.6}] This is just an immediate consequence of
  (\ref{3.1}) and Theorem \ref{mainresult}.
\end{proof}

\begin{proof}[Proof of Corollary \ref{second}] 
 Let $N=\lceil C\rceil +n+1$. It follows from Theorem \ref{domymath} that $S_{N,A}$ is a priori complete pluripolar subset. We claim that
\begin{equation}\label{calimm}
S_{N,A}=Z_1 (\varphi):=\bigcup_{y\in Y}\{x\in X_y; c_x (\varphi |_{X_y})\leq 1\},
\end{equation}
from which analyticity of $S_{N,A}$ follows immediately by Theorem \ref{third} and the fact that $\varphi$ has analytic singularities. 

By the extremal properties of Bergman kernels, in order to prove (\ref{calimm}), it will be sufficient to show the "$\subset$" direction. In fact, according to Siu's uniform global generation formula (cf. Theorem 6.27 in \cite{demailly-book-hep} or Proposition~1 in \cite{siuin}), for any pseudoeffective line bundle $(L,h)$ on $X_y$ and $y\in Y$, the sheaf $\mathcal{O}(K_{X_y}+(n+1)A+L)\otimes\mathcal{I}(h)$ is generated by its global sections. Take $(L,h)=(\lceil C\rceil A, h_A^{\lceil C\rceil}e^{-2\varphi})$, then we obtain once $\mathcal{I}(\varphi|_{X_y})_p=\mathcal{O}_{X_y}$ for some $p\in X_y$, $K_{X/Y,NA,h_{A}^{N}e^{-2\varphi}}$ will be positive at $p$ from (\ref{chaoshi}). Hence the left hand side in (\ref{calimm}) is contained in the right hand side and the proof is complete now.
\end{proof}

\section{Some counterexamples}\label{sec6}

In this section, we turn to analyze some explicit examples of the
aforementioned subsets. To show that $X_c (\varphi)$ and $Y_c (\varphi)$ are
in general non-analytic, the authors of {\cite{wangxiaoqin}} have constructed
the following example:

\begin{example}\label{buyong}
  (Example 4.3 in {\cite{wangxiaoqin}}) Let
$$
    \varphi (z, w) := \sum_{k = 1}^{\infty} \log (| w - w_k - z^{m_k}
    |^{\alpha_k} + | z |^{\beta_k})
$$
  for $(z, w) \in \Delta^2$. We might choose $w_k =(k + 1)^{-1}$ and
  $\alpha_k = k^{- 2}$, $m_k$ and $\beta_k$ to be determined later. Thus
  $\varphi (z, w)$ has a minorant
$$
    \sum_{k = 1}^{\infty} \alpha_k \log | w - w_k - z^{m_k}  |  \in
    L_{\text{loc}}^1 (\Delta^2)  
$$
  and consequently, is well-defined. It is clear that $\varphi (z, w)$ is not
  continuous at $(0, 0)$ since $\varphi (0, 0) > - \infty$ but $\varphi(0,w_k)=-\infty$. A standard
  computation for Lelong numbers yields that:
$$
    \nu (\varphi_{w_k}, (0,w_k)) = \min (m_k \alpha_k, \beta_k) .
$$
  So if we take for instance $m_k = c \cdot k^2$ and $\beta_k = c + 1$, then
  all $(0, w_k) \in X_c (\varphi)$ but $(0, 0) \notin X_c (\varphi)$. The same
  argument also holds for $Y_c (\varphi)  .$
\end{example}

Before giving the construction of \emph{continuous} plurisubharmonic counterexamples, let us recall the
definition of generalized Cantor sets.

\begin{definition}
  Let $\{s_k \}_{k \geq 1}$ be a sequence of numbers such that $0 < s_k < 1$
  for all $k$. Denote $C (s_1)$ to be the set obtained from $[0, 1]$ by
  removing an open interval with length $s_1$ from the center. By induction,
  let $C (s_1, \ldots, s_k)$ be obtained by removing an open interval whose
  length is a proportion $s_k$ of $C (s_1, \ldots, s_{k - 1})$ from each
  center of small intervals of $C (s_1, \ldots, s_{k - 1})$. Then the
  associated generalized Cantor set is defined to be
$$
    \mathcal{C}= \bigcap_{k = 1}^{\infty} C (s_1, \ldots s_k) .
$$
\end{definition}

Let's recall the construction of such $\varphi$ in \cite{lichi}, which aims at
giving a counterexample to Demailly's approximation conjecture of
plurisubharmonic functions with isolated singularities.

\begin{example}
  {{{{(\cite{lichi})}} }}Let $\mathcal{O}_{\mathbb{P}^1} (- 1)
  := M$ be the tautological line bundle over $\mathbb{P}^1$, or
  equivalently, the blow up of $\mathbb{C}^2$ at the origin. Let $\rho : M
  \rightarrow \mathbb{P}^1$ and $\theta : M \rightarrow \mathbb{C}^2$ be the
  corresponding mappings. If we denote the canonical projection map by $\pi :
  \mathbb{C}^2 \backslash \{ 0 \} \rightarrow \mathbb{P}^1$, then those three
  mappings satisfies: $\rho \circ \theta^{- 1} = \pi$ on $\mathbb{C}^2
  \backslash \{ 0 \}$. Let $\omega_0$ be the normalized ($\int_{\mathbb{P}^1}\omega_0 =1$) Fubini-Study $(1, 1)$
  form defined on $\mathbb{P}^1$, and $u$ a $\omega_0$-psh function to be
  determined.
  
  The classical potential theory implies that there is an one-to-one
  correspondence between the space of $\omega_0$-psh functions and the space
  of probability measures in $\mathbb{P}^1$, which are respectively given by:
$$
    u\in\{\text{$\omega_0$-psh functions}\} \mapsto w_0 + \sqrt{- 1} \partial \overline{\partial} u\in\{\text{probability measures in $\mathbb{P}^1$}\}
$$
  and
$$
    \mu\in\{\text{probability measures in $\mathbb{P}^1$}\} \mapsto 
    p_{\mu}\in\{\text{$\omega_0$-psh functions}\}
$$
  where $p_{\mu}:=2 \pi \int_{\mathbb{P}^1} (- G (z, w)) d \mu (w)$, $G(\cdot,\cdot)$ is the Green function of the Laplace operator of $\omega_0$ (cf. section 3 in \cite{lichi}). One may verify that $w_0 + \sqrt{- 1} \partial \overline{\partial} p_u =
  \mu$.
  
  Now choose a non-decreasing generalized Cantor function $c (x)$ (see (9) in
  \cite{lichi}) associated to any fixed generalized Cantor set
  $\mathcal{C}=\mathcal{C} (s_1, \ldots ., s_k)$ as a distribution function,
  and denote the measure associated to $c (x)$ by $\mu$. Continuity of $c (x)$ implies that $\mu$ has no atoms at any point of
  $\mathcal{C}.$ Let $p_{\mu}$ be the potential of $\mu$, then it is easy to
  see that $p_{\mu}$ is smooth on $\mathbb{P}^1 \backslash \mathcal{C}$, and
  precisely takes the value $- \infty$ at $\mathcal{C}$. If the choice of $\{
  s_k \}$ satisfies certain conditions (Proposition 3.3 in \cite{lichi}),
  then $p_{\mu}$ will be continuous as a function from $\mathbb{P}^1$ to $[-
  \infty, + \infty)$.
  
  Set
  \begin{equation}
    \varphi := \theta^{\ast} ( \max \{ 2 \log | z
    |^2, \log | z |^2 + \pi^{\ast} p_{\mu} \}) \label{4.6}, | z |^2 =
    |z_1 |^2 + |z_2 |^2 
  \end{equation}
  to be a continuous plurisubharmonic function on $M$ by adding
  value $- \infty$ to the exceptional divisor $\mathbb{P}^1$.
\end{example}

\begin{theorem}
  Let $\varphi$ be defined as in {{(\ref{4.6})}}. Fix a small unit polydisc $\Delta_z\times\Delta_w\subset M $ neighborhood of $(0_o, o)$ so that $\rho(z,w)=w$, then $X_c (\varphi)$
  {{(and equivalently $Y_c (\varphi)$, which are defined as subsets of $\Delta_z\times\Delta_w$)}} is not analytic at $(0_o, o)$,
  where $w=o$ is the origin in $\mathbb{C}\subset\mathbb{P}^1$.
\end{theorem}

\begin{proof}
  Let us first note that the line bundle $M$ is trivialized as $\mathbb{C}\times \Delta_w$ when restricted as a fibration over $\Delta_w \subset \mathbb{P}^1$, then for any $w \in
  \mathcal{C}$, $\nu (\varphi_w, (0_w, w)) = 4.$ However, since $p_{\mu}$ is
  smooth on $\mathbb{P}^1 \backslash \mathcal{C}$ we know that for any $w \in
  \Delta_w \backslash \mathcal{C}$, $\nu (\varphi_s, (0_w, w)) = 2.$ This already
  shows the non-analyticity of the germ of $X_3 (\varphi)$ (and $Y_{1 / 3}
  (\varphi)$) at $(0_o,o)$ because $X_3 (\varphi)  \bigcap \Delta_w =\mathcal{C}$ is a generalized
  Cantor set.
\end{proof}

\begin{acknowledgements}
The author would sincerely thank his supervisor Prof. Xiangyu Zhou for the valuable comments and discussions, and the anonymous referee(s) for his/her suggestions to the original version of this paper. Additionally, the author is extremely grateful to Dr. Zhenqian Li for the fruitful suggestions and Prof. Qi'an Guan for the generous support.
\end{acknowledgements}


\begin{thebibliography}{1}
  
  \bibitem{as}Angehrn, U. and  Siu, Y.-T., {\newblock}\emph{Effective freeness and
  point separation for adjoint bundles}, {\newblock}{Invent. Math.}
  {122} (1995), 291--308.{\newblock}
  
   \bibitem{ber}Berndtsson, B., \emph{Curvature of vector bundles associated to holomorphic fibrations},
  {Ann. of Math.}  {169} (2009), no.~2, 531--560.
  
  \bibitem{berndtsson2}Berndtsson, B., {\newblock}\emph{Subharmonicity properties
  of the Bergman kernel and some other functions associated to pseudoconvex
  domains},  {\newblock} {Ann. Inst. Fourier (Grenoble)} 
   {56} (2006), 1633--1662.{\newblock}
  
  \bibitem{bern13}Berndtsson, B., \emph{{\newblock}The openness conjecture for
  plurisubharmonic functions},{\newblock}  { arXiv:1305.5781v1}.
  
  \bibitem{bp08}Berndtsson, B. and  Păun, M., \emph{Bergman kernels and the pseudoeffectivity of relative canonical bundles},
  {Duke Math. J.}  {145} (2008), no.~2, 341--378.
  
  \bibitem{bl}Bierstone, E. and Milman, P. D., \emph{A simple constructive proof of canonical resolution of singularities}, in  {Effective methods in algebraic geometry} (Castiglioncello, 1990), edited by T. Mora and C. Traverso, Progr. Math.  {94}, Birkh\"{a}user, Boston (1991), pp. 11–30.
  
  \bibitem{demailly-book-hep}Demailly, J.-P., {\newblock} \emph{Analytic
  methods in algebraic geometry}, {\newblock}Surv. Mod. Math. 1, International
  Press, Somerville (2012).{\newblock}
  
  \bibitem{demailly-book}Demailly, J.-P., {\newblock} \emph{Complex
  analytic and differential geometry}, { {electronically}},
  {\newblock}http://www-fourier.ujf-grenoble.fr/~demailly/books.html.{\newblock}
  
  \bibitem{dem}Demailly, J.-P., Hwang, J.-M. and Peternell, T., \emph{Compact manifolds covered by a torus},
 { J. Geom. Anal.}  {18} (2008), no.~2, 324-340.
  
  \bibitem{demailly-kollar}Demailly, J.-P. and Koll{\'a}r, J.,
  {\newblock}\emph{Semi-continuity of complex singularity exponents and
  K{\"a}hler-Einstein metrics on Fano orbifolds}, {\newblock} {Ann.
  Sci. {\'E}cole Norm. Sup.}  {34} (2001), 525--556.{\newblock}
    
  \bibitem{guanli}Guan, Q. and Li, Z., {\newblock}\emph{Analytic adjoint
  ideal sheaves associated to plurisubharmonic functions},
  {\newblock} {Ann. Sc. Norm. Super. Pisa Cl. Sci.},
   {18} (2018), no.~1, 391--395.{\newblock}
  
  \bibitem{guan-zhou-li}Guan, Q., Li, Z. and Zhou, X.,
  {\newblock}\emph{Estimation of weighted $L^2$ norm related to Demailly's strong
  openness conjecture}, {\newblock} { arXiv:1603.05733}.{\newblock}
  
  \bibitem{guan-zhou15soc}Guan, Q., and  Zhou, X., {\newblock}\emph{A proof of
  Demailly's strong openness conjecture}, {\newblock} {Ann. of Math.}
   {182} (2015), no.~2, 605--616.{\newblock}
  
  \bibitem{guan-zhou1}Guan, Q., and  Zhou, X., {\newblock}\emph{Lelong
  numbers, complex singularity exponents, and Siu's semicontinuity theorem},
  {\newblock} {C. R. Acad. Sci. Paris, Ser. I}
   {355} (2017), 415--419.{\newblock}
  
  \bibitem{guancrelle}Guan, Q., and  Zhou, X., {\newblock}\emph{Restriction
  formula and subadditivity property related to multiplier ideal sheaves},
  {\newblock} {J. Reine Angew. Math.} {769} (2020), 1--33.{\newblock}
  
  \bibitem{guenancia2}Guenancia, H., {\newblock}\emph{Toric plurisubharmonic
  functions and analytic adjoint ideal sheaves}, {\newblock} {Math.
  Z.}  {271} (2012), no.~3-4, 1011--1035.{\newblock}
  
  \bibitem{hart}Hartshorne, R., \emph{Algebraic geometry},
 {Graduate Texts in Mathematics}, no. 52, Springer-Verlag (1977).
  
  \bibitem{hironaka}Hironaka, H., {\newblock}\emph{Resolutions of singlarities of
  an algebraic variety over a field of characteristic zero, I and II},
  {\newblock} {Ann. of Math.}  {79} (1964), no.~2, 109-326.{\newblock}
  
  \bibitem{hormanderbook}H{\"o}rmander, L., {\newblock}\emph {An
  Introduction to Complex Analysis in Several Variables}, {\newblock}ASR
  edition, Elsevier Science Publishers, New York, 1966, 3rd revised edition,
  North-Holland Math. Library, vol. 7, North Holland, Amsterdam,
  1990.{\newblock}
  
  
  \bibitem{kis}Kiselman, C. O., {\newblock}\emph{The partial legendre
  transformation for plurisubharmonic functions}, {\newblock} {Invent.
  Math.}  {49} (1978), no.~2, 137--148.{\newblock}
  
  \bibitem{kiselman}Kiselman, C. O., {\newblock}\emph{Attenuating the
  singularities of plurisubharmonic functions}, {\newblock} {Ann. Pol.
  Math.}  {60} (1994), 173--197.{\newblock}
  
  \bibitem{Larbook}Lazarsfeld, R., {\newblock} \emph{Positivity in
  Algebraic Geometry I and II}, {\newblock}Ergebnisse der Mathematik und
  ihrer Grenzgebiete. 3. Folge (2004), Springer-Verlag, Berlin, 48-49.{\newblock}
  
  \bibitem{lichi}Li, C., {\newblock}\emph{Analytical approximations and Monge
  Amp{\`e}re masses of plurisubharmonic singularities},
  {\newblock} { arXiv:2012.15599v1}.{\newblock}
  
  \bibitem{ohsawa-takegoshi}Ohsawa, T.  and  Takegoshi, K., \emph{{\newblock}On
  the extension of $L^2$ holomorphic functions}, {\newblock} {Math.
  Z.}  {195} (1987), 197--204.{\newblock}
  
  \bibitem{phongsturm}Phong, D. H. and  Sturm, J.,  {\newblock}\emph{Algebraic
  estimates, stability of local zeta functions, and uniform estimates for
  distribution functions}, {\newblock} {Ann. of Math.}
   {152} (2000), no.1, 277--329.{\newblock}
  
  \bibitem{siu2}Siu, Y.-T., {\newblock}\emph{Analyticity of sets associated to
  Lelong numbers and the extension of closed positive currents},
  {\newblock} {Invent. Math.} {27} (1974), 53--156.{\newblock}
  
  \bibitem{siuin}Siu, Y.-T., \emph{Invariance of plurigenera},
  {Invent. Math.}  {134} (1998), no.~3, 661--673.
  
  \bibitem{varchenko}Varchenko, A. N., {\newblock}\emph{The complex singularity
  index of a singularity do not change along the stratum $\mu =$ constant},
  {\newblock} {Functional Anal. Appl.}  {16} (1982), no.~1, 1--9.{\newblock}
  
  \bibitem{var}Varouchas, J., \emph{Stabilité de la classe des variétés kählériennes par certains morphismes propres},  {Invent.
Math.}  {77} (1984), no.~1, 117–127.
  
  \bibitem{wangxiaoqin}Wang, X. Q., {\newblock}\emph{Analyticity theorems for
  parameter-dependent currents}, {\newblock} {Math. Scand.}
   {69} (1991), 179--198.{\newblock}
  
  \bibitem{xiamingchen}Xia, M. C., \emph{Analytic Bertini theorem}, Math. Z. 302 (2022), no.~2, 1171-1176.
  
  \bibitem{zhouzhujdg}Zhou, X., and  Zhu, L., \emph{An optimal $L^{2}$ extension theorem on weakly pseudoconvex Kähler manifolds},
  {J. Differential Geom.}  {110} (2018), no.~1, 135--186.
  
  \bibitem{zhouzhu2}Zhou, X., and  Zhu, L., {\newblock}\emph{A generalized
  Siu's lemma}, {\newblock} {Math. Res. Lett.}
   {24} (2017), no.~6, 1897--1913.{\newblock}
  
  \bibitem{zhouzhu}Zhou, X., and  Zhu, L., {\newblock}\emph{Siu's lemma,
  optimal $L^2$ extension and applications to twisted pluricanonical sheaves},
  {\newblock} {Math. Ann.}  {377} (2020), no.~1-2, 675--722.{\newblock}

\end{thebibliography}
\end{document}